\documentclass[reqno,a4paper]{amsart}
\usepackage{fancyhdr}
\usepackage{etoolbox}
\makeatletter
\let\old@tocline\@tocline
\let\section@tocline\@tocline
\newcommand{\subsection@dotsep}{4.5}
\newcommand{\subsubsection@dotsep}{4.5}
\patchcmd{\@tocline}
  {\hfil}
  {\nobreak
     \leaders\hbox{$\m@th
        \mkern \subsection@dotsep mu\hbox{.}\mkern \subsection@dotsep mu$}\hfill
     \nobreak}{}{}
\let\subsection@tocline\@tocline
\let\@tocline\old@tocline

\patchcmd{\@tocline}
  {\hfil}
  {\nobreak
     \leaders\hbox{$\m@th
        \mkern \subsubsection@dotsep mu\hbox{.}\mkern \subsubsection@dotsep mu$}\hfill
     \nobreak}{}{}
\let\subsubsection@tocline\@tocline
\let\@tocline\old@tocline

\let\old@l@subsection\l@subsection
\let\old@l@subsubsection\l@subsubsection

\def\@tocwriteb#1#2#3{%
  \begingroup
    \@xp\def\csname #2@tocline\endcsname##1##2##3##4##5##6{%
      \ifnum##1>\c@tocdepth
      \else \sbox\z@{##5\let\indentlabel\@tochangmeasure##6}\fi}%
    \csname l@#2\endcsname{#1{\csname#2name\endcsname}{\@secnumber}{}}%
  \endgroup
  \addcontentsline{toc}{#2}%
    {\protect#1{\csname#2name\endcsname}{\@secnumber}{#3}}}%

\newlength{\@tocsectionindent}
\newlength{\@tocsubsectionindent}
\newlength{\@tocsubsubsectionindent}
\newlength{\@tocsectionnumwidth}
\newlength{\@tocsubsectionnumwidth}
\newlength{\@tocsubsubsectionnumwidth}
\newcommand{\settocsectionnumwidth}[1]{\setlength{\@tocsectionnumwidth}{#1}}
\newcommand{\settocsubsectionnumwidth}[1]{\setlength{\@tocsubsectionnumwidth}{#1}}
\newcommand{\settocsubsubsectionnumwidth}[1]{\setlength{\@tocsubsubsectionnumwidth}{#1}}
\newcommand{\settocsectionindent}[1]{\setlength{\@tocsectionindent}{#1}}
\newcommand{\settocsubsectionindent}[1]{\setlength{\@tocsubsectionindent}{#1}}
\newcommand{\settocsubsubsectionindent}[1]{\setlength{\@tocsubsubsectionindent}{#1}}

\renewcommand{\l@section}{\section@tocline{1}{\@tocsectionvskip}{\@tocsectionindent}{}{\@tocsectionformat}}%
\renewcommand{\l@subsection}{\subsection@tocline{1}{\@tocsubsectionvskip}{\@tocsubsectionindent}{}{\@tocsubsectionformat}}%
\renewcommand{\l@subsubsection}{\subsubsection@tocline{1}{\@tocsubsubsectionvskip}{\@tocsubsubsectionindent}{}{\@tocsubsubsectionformat}}%
\newcommand{\@tocsectionformat}{}
\newcommand{\@tocsubsectionformat}{}
\newcommand{\@tocsubsubsectionformat}{}
\expandafter\def\csname toc@1format\endcsname{\@tocsectionformat}
\expandafter\def\csname toc@2format\endcsname{\@tocsubsectionformat}
\expandafter\def\csname toc@3format\endcsname{\@tocsubsubsectionformat}
\newcommand{\settocsectionformat}[1]{\renewcommand{\@tocsectionformat}{#1}}
\newcommand{\settocsubsectionformat}[1]{\renewcommand{\@tocsubsectionformat}{#1}}
\newcommand{\settocsubsubsectionformat}[1]{\renewcommand{\@tocsubsubsectionformat}{#1}}
\newlength{\@tocsectionvskip}
\newcommand{\settocsectionvskip}[1]{\setlength{\@tocsectionvskip}{#1}}
\newlength{\@tocsubsectionvskip}
\newcommand{\settocsubsectionvskip}[1]{\setlength{\@tocsubsectionvskip}{#1}}
\newlength{\@tocsubsubsectionvskip}
\newcommand{\settocsubsubsectionvskip}[1]{\setlength{\@tocsubsubsectionvskip}{#1}}

\patchcmd{\tocsection}{\indentlabel}{\makebox[\@tocsectionnumwidth][l]}{}{}
\patchcmd{\tocsubsection}{\indentlabel}{\makebox[\@tocsubsectionnumwidth][l]}{}{}
\patchcmd{\tocsubsubsection}{\indentlabel}{\makebox[\@tocsubsubsectionnumwidth][l]}{}{}

\newcommand{\@sectypepnumformat}{}
\renewcommand{\contentsline}[1]{%
  \expandafter\let\expandafter\@sectypepnumformat\csname @toc#1pnumformat\endcsname%
  \csname l@#1\endcsname}
\newcommand{\@tocsectionpnumformat}{}
\newcommand{\@tocsubsectionpnumformat}{}
\newcommand{\@tocsubsubsectionpnumformat}{}
\newcommand{\setsectionpnumformat}[1]{\renewcommand{\@tocsectionpnumformat}{#1}}
\newcommand{\setsubsectionpnumformat}[1]{\renewcommand{\@tocsubsectionpnumformat}{#1}}
\newcommand{\setsubsubsectionpnumformat}[1]{\renewcommand{\@tocsubsubsectionpnumformat}{#1}}
\renewcommand{\@tocpagenum}[1]{%
  \hfill {\mdseries\@sectypepnumformat #1}}

\let\oldappendix\appendix
\renewcommand{\appendix}{%
  \leavevmode\oldappendix%
  \addtocontents{toc}{%
    \protect\settowidth{\protect\@tocsectionnumwidth}{\protect\@tocsectionformat\sectionname\space}%
    \protect\addtolength{\protect\@tocsectionnumwidth}{2em}}%
}
\makeatother

\makeatletter
\settocsectionnumwidth{2em}
\settocsubsectionnumwidth{2.5em}
\settocsubsubsectionnumwidth{3em}
\settocsectionindent{1pc}%
\settocsubsectionindent{\dimexpr\@tocsectionindent+\@tocsectionnumwidth}%
\settocsubsubsectionindent{\dimexpr\@tocsubsectionindent+\@tocsubsectionnumwidth}%
\makeatother

\settocsectionvskip{10pt}
\settocsubsectionvskip{0pt}
\settocsubsubsectionvskip{0pt}

\settocsectionformat{\bfseries}
\settocsubsectionformat{\mdseries}
\settocsubsubsectionformat{\mdseries}
\setsectionpnumformat{\bfseries}
\setsubsectionpnumformat{\mdseries}
\setsubsubsectionpnumformat{\mdseries}

\let\oldtableofcontents\tableofcontents
\renewcommand{\tableofcontents}{%
  \vspace*{-\linespacing}
  \oldtableofcontents}

\usepackage{times}
\usepackage{mathrsfs} 
\usepackage{latexsym,amsopn,amssymb,amsmath}
\usepackage{amsthm}
\usepackage{amsfonts,amsbsy,amscd,stmaryrd}
\usepackage{paralist}
\usepackage{enumitem} 
\usepackage{bm}
\usepackage{mathrsfs} 
\usepackage[pagebackref,hypertexnames=true]{hyperref}
\hypersetup{colorlinks=true,linkcolor=RoyalBlue,%
urlcolor=RoyalBlue,citecolor=RoyalBlue}
\usepackage[usenames,dvipsnames,svgnames,table]{xcolor}

\newcommand{\C}{\mathbb{C}}

\newcommand{\N}{\mathbb{N}}

\newcommand{\R}{\mathbb{R}}

\newcommand{\Z}{\mathbb{Z}}
\newcommand{\mbs}{\mathbb{S}}
\newcommand{\ca}{\mathcal{A}}

\newcommand{\ce}{\mathcal{E}}
\newcommand{\cf}{\mathcal{F}}
\newcommand{\cg}{\mathcal{G}}
\newcommand{\ch}{\mathcal{H}}

\newcommand{\cl}{\mathcal{L}}
\newcommand{\cm}{\mathcal{M}}
\newcommand{\cn}{\mathcal{N}}

\newcommand{\cR}{\mathcal{R}}
\newcommand{\cs}{\mathcal{S}}

\newcommand{\cx}{\mathcal{X}}
\def\rond{\mathscr}
\newcommand{\ra}{\rond{A}}
\newcommand{\rb}{\rond{B}}
\newcommand{\rc}{\rond{C}}
\newcommand{\rd}{\rond{D}}
\newcommand{\re}{\rond{E}}

\newcommand{\rk}{\rond{K}}

\newcommand{\rr}{\rond{R}}

\newcommand{\rx}{\rond{X}}
\newcommand{\op}{\mathrm{op}}
\def\diam{\mathrm{diam}}
\def\dist{\mathrm{dist}}

\def\rmb{\mathrm{b}}
\def\rmc{\mathrm{c}}

\def\d{\mathrm{d}}

\def\e{\mathrm{e}}

\def\rmi{\mathrm{i}}

\def\rmu{\mathrm{u}}

\def\veps{\varepsilon}
\def\vphi{\varphi}
\def\vkappa{\varkappa}

\def\braket#1#2{\langle{#1}|{#2}\rangle}

\def\rarrow{\rightarrow}

\def\what{\widehat}
\def\what#1{\widehat{ #1\,}}

\DeclareMathOperator*{\slim}{s-lim}
\DeclareMathOperator*{\wlim}{w-lim}
\DeclareMathOperator*{\ulim}{u-lim}
\def\nin{\notin}
\def\supp{\mbox{\rm supp\! }}

\def\loc{\mathrm{loc}}

\def\spec{\mathrm{Sp}}
\def\spe{\mathrm{Sp_{ess}}}

\def\nin{\notin}

\def\ccup{\textstyle{\bigcup}}

\newcommand{\oX}{\overline{X}}

\newcommand{\oXY}{\overline{X/Y}}
\newcommand{\oXZ}{\overline{X/Z}}

\def\ol#1{\overline{#1}}

\usepackage{perpage}
\MakePerPage{footnote}

\newtheorem{theorem}{Theorem}[section]
\newtheorem{lemma}[theorem]{Lemma}
\newtheorem{proposition}[theorem]{Proposition}
\newtheorem{corollary}[theorem]{Corollary}

\newtheorem{remark}[theorem]{Remark}

\newtheorem{example}[theorem]{Example}

\vspace*{40mm}

\title[Essential Spectrum]
{On the essential spectrum \\
of elliptic differential operators  
\\[2mm]
{\tiny J\MakeLowercase{anuary} 3, 2018} }

\begin{document}

\author[V. Georgescu]{Vladimir Georgescu}
     \address{V. Georgescu, D\'epartement de Math\'ematiques,
       Universit\'e de Cergy-Pontoise, 95000 Cergy-Pontoise, France}
     \email{vladimir.georgescu@math.cnrs.fr}

\thanks{{\em AMS Subject classification (2010):} 35P05 (Primary) 
      34L05, 47A10, 47L65, 58J50, 81Q10 (Secondary).}
\thanks{{\em Key words and phrases:} Spectral analysis, essential
  spectrum, Fredholm, $C^*$-algebra, crossed product. }

\begin{abstract} 
  Let $\ca$ be a $C^*$-algebra of bounded uniformly continuous
  functions on a finite dimensional real vector space $X$ such that
  $\ca$ is stable under translations and contains the continuous
  functions that have a limit at infinity. Denote $\ca^\dagger$ the
  boundary of $X$ in the character space of $\ca$. Then to each
  operator $A$ in the crossed product $\ca\rtimes X$ one may naturally
  associate a family of bounded operators $A_\vkappa$ on $L^2(X)$
  indexed by the characters $\vkappa\in\ca^\dagger$. We show that the
  essential spectrum of $A$ is the union of the spectra of the
  operators $A_\vkappa$.  The applications cover very general classes
  of singular elliptic operators.
\end{abstract}

\maketitle 

\newpage

\tableofcontents

\newpage

\section{Introduction}\label{s:intro}

\subsection{Elliptic algebra}\label{ss:EllAlg}

Let $X=\R^d$ and $L^2=L^2(X)$.  We denote $\rb=\rb(X)$ the algebra of
bounded operators on $L^2$ and $\rk=\rk(X)$ that of compact operators.
Let $\rb_\loc$ be the space $\rb$ equipped with the \emph{local norm
  topology} defined by the family of seminorms
$\|A\|_\theta=\|A\theta(q)\|$ where $\theta\in C_0(X)$ (continuous
functions which tend to zero at infinity) and $\theta(q)$ means
multiplication by $\theta$.  Clearly this topology is metrizable and
finer than the strong operator topology.  If $\theta(x)>0\ \forall x$
then $\|\cdot\|_\theta$ is a norm on $\rb$ which on bounded subsets
defines the local norm topology.  If $(A_s)_s$ is a sequence which
converges in $\rb_\loc$ to $A$, we say that the sequence is
\emph{locally norm convergent} and write $\ulim_s A_s=A$.

If $a\in X$ then $\e^{\rmi aq}$ and $\e^{\rmi ap}$ are the unitary
operators on $L$ that act as follows:
\begin{equation}\label{eq:qp}
(\e^{\rmi aq}u)(x)=\e^{\rmi ax}u(x) 
\quad\text{and}\quad
(\e^{\rmi ap}u)(x)=u(x+a).
\end{equation}
We also use an alternative notation for the translation by $a\in X$ of
a function, namely $\tau_a(\varphi)(x)=\varphi(a+x)$, and extend it to
operators: $\tau_a(A)\doteq\e^{\rmi ap}A\e^{-\rmi ap}$ for $A\in\rb$.
The \emph{elliptic algebra of $X$} (the name will be justified page
\pageref{p:caell}) is defined by
\begin{equation}\label{eq:ellalg}
\re=\{A\in\rb \mid \lim_{a\to0}\|(\e^{\rmi ap}-1)A^{(*)}\|=0,\
\lim_{a\to0}\|\e^{-\rmi aq}A\e^{\rmi aq}-A\|=0 \} .
\end{equation}
The notation $A^{(*)}$ means that the relation must hold for both $A$
and $A^*$. Clearly $\re$ is a $C^*$-algebra. $\re_\loc$ is the set
$\re$ equipped with the local norm topology inherited from $\rb_\loc$.

Our main result requires more formalism but we can state right now the
simplest particular case, which does not require any $C^*$-algebra
background.  We denote $X^\dagger$ the set of all ultrafilters finer
than the Fr\'echet filter on $X$.  We denote $\spec(A)$ the spectrum
and $\spe(A)$ the essential spectrum of an operator $A$.

\begin{theorem}\label{th:rb}
  If $A\in\re$ then the limit
  $\ulim_{x\to\vkappa}\tau_x(A)\doteq A_\vkappa$ exists
  $\forall\vkappa\in X^\dagger$ and
\begin{equation}\label{eq:rb}
\spe (A)= \ccup_{\vkappa\in X^\dagger} \spec(A_\vkappa).
\end{equation} 
\end{theorem}

\textbf{Remark.} That the convergence holds locally in norm is
important for the proof of the theorem. This type of convergence has
been used in \cite{LaS}, see for example (4.24) there.

For the convenience of the reader, we recall in an appendix \S\ref{s:app}
some facts concerning filters and ultrafilters and also reformulate
Theorem \ref{th:rb} in terms of sequences, as in \cite{LS,RRS0,RRS}
for example. We mention that ultrafilters have been first used in this
context in \cite[Th.\ 4.1]{GI1} and then in \cite{R1}. In
\S\ref{s:app} we will see that ultrafilters play a quite
natural r\^ole in the theory.

In \S\ref{ss:unbounded} we will extend this result to the unbounded
operators whose resolvent belongs to $\re$. We mention a simpler
result in the self-adjoint case. Note that non-densely defined
self-adjoint operators could appear as limits. For example, quite
often $H_\vkappa=\infty$ where $\infty$ is the operator with $\{0\}$
as domain and $0$ as resolvent. Clearly, $H$ has purely discrete
spectrum, i.e.\ $\spe(H)=\emptyset$, if and only if $H_\vkappa=\infty$
for all $\vkappa$. See also \S\ref{sss:unbounded}.

\begin{corollary}\label{co:classic}
  Let $H$ be a self-adjoint operator on $L^2$ such that
  $R(z)=(H-z)^{-1}$ satisfies for some, hence for all, $z$ in the
  resolvent set of $H$:
  \begin{equation}\label{eq:sadj}
    \lim_{a\to0}\|(\e^{\rmi ap}-1) R(z)\|=0 \quad\text{and}\quad 
    \lim_{a\to0}\|[\e^{\rmi aq}, R(z)]\|=0 .
  \end{equation}
  If $\vkappa\in X^\dagger$ then
  $\ulim_{x\to\vkappa}\tau_x(R(z))\doteq R_\vkappa(z)$ exists for all
  $z$ in the resolvent set of $H$ and $R_\vkappa(z)$ is a self-adjoint
  pseudo-resolvent, hence is the resolvent of a self-adjoint operator
  $H_\vkappa$ acting in a closed subspace $L^2_\vkappa$ (which could
  be reduced to $0$) of $L^2$. We have
  \begin{equation}\label{eq:rb-aa}
\spe (H)= \ccup_{\vkappa\in X^\dagger} \spec(H_\vkappa)
\end{equation}
where the spectrum of $H_\vkappa$ is computed in the subspace
$L^2_\vkappa$. 
\end{corollary}

This gives a complete proof of Corollary 4.2 in \cite{GI1}: the proof
sketched on page 31 there gives \eqref{eq:rb-aa} but with union
replaced by the closure of the union \cite[Th.\ 1.2]{GI4}.

\subsection{General algebras.}\label{ss:alg}

This subsection contains some preliminary notations and material
required for the presentation of our main result Theorem \ref{th:GIA}.

$C_\rmb^\rmu(X)$ is the algebra of bounded uniformly continuous
functions on $X$ and the subalgebras consisting of functions which
have compact support, or tend to zero at infinity, or have a limit at
infinity are denoted $C_\rmc(X)$, $C_0(X)$, and
$C_\infty(X)=C_0(X)+\C$ respectively.

A complex measurable function $\varphi$ on $X$ is usually identified
with the operator of multiplication by $\varphi$ on $L^2$, but for
clarity it is sometimes convenient to denote $\varphi(q)$ this
operator. Then we define $\vphi(p)=\cf^{-1}\varphi(q)\cf$, where $\cf$
is the Fourier transform (formally $p=-\rmi\nabla$).  With the
notation $\tau_a$ introduced before:
$\tau_a(\varphi(q))=(\tau_a(\varphi))(q)$ and
$\tau_a(\varphi(p))=\varphi(p)$.

Let $\ca$ be a $C^*$-algebra of bounded uniformly continuous functions
on $X$ such that $\ca$ is \emph{stable under translations} and
contains the set of continuous functions that have a limit at
infinity: so $C_\infty(X)\subset\ca\subset C_\rmb^\rmu(X)$.  The
crossed product $\ra=\ca\rtimes X$ of the $C^*$-algebra $\ca$ by the
action of $X$ \cite{W} is canonically isomorphic with the
$C^*$-algebra of operators on $L^2$ defined as the norm closed linear
subspace of $\rb$ generated by the operators $\varphi(q)\psi(p)$ with
$\varphi\in\ca$ and $\psi\in C_0(X)$.  Obviously $C_0(X)\rtimes X=\rk$
and $\rk\subset\ra$ for any $\ca$.

One may describe $\ra$ as a \emph{$C^*$-algebra of elliptic
  operators}: if $h$ is a real elliptic polynomial of order $m$ on
$X$, then $\ra$ is the $C^*$-algebra generated%
\footnote{ In the sense that it is the smallest $C^*$-algebra which
  contains the resolvents of these operators.}
by the self-adjoint operators of the form $h(p) + S$, where $S$ runs
over the set of symmetric differential operators of order $< m$ with
coefficients in
$\ca^\infty=\{\varphi\in C^\infty(X)\mid \varphi^{(\alpha)}\in\ca\
\forall\alpha\}$ \cite[Cor.\ 2.4]{DG0}. \label{p:caell}

The largest algebra $\ca$ allowed by our conditions is
$C_\rmb^\rmu(X)$ and \emph{the corresponding crossed product coincides
  with the elliptic algebra}: $C_\rmb^\rmu(X)\rtimes X=\re$
\cite{G,GI3} (then $\ca^\infty=C^\infty_\rmb(X)$). $\re$ has another
interesting description which makes the connection with the
band-dominated operators \cite{RRS}. Say that $k:X^2\to\C$ is a
\emph{controlled kernel} if there is $r>0$ real such that $k(x,y)=0$
if $|x-y|>r$ \cite{R}. If $k$ is a bounded uniformly continuous
controlled kernel then it defines a bounded integral operator and
$\re$ is the norm closure of the set of such operators
\cite[Prop. 6.5]{G}.  This characterization is independent of the
group structure of $X$.

The \emph{character space}, or \emph{spectrum}, of $\ca$ is the
compact space $\sigma(\ca)$ consisting of nonzero morphisms $\ca\to\C$
equipped with the weak$^*$ topology inherited from the embedding
$\sigma(\ca)\subset\ca'$ (dual of $\ca$). Each $x\in X$ defines a
character $\chi_x:\varphi\mapsto\varphi(x)$ and the map
$x\mapsto\chi_x$ is a homeomorphism of $X$ onto an open dense subset
of $\sigma(\ca)$ that we identify with $X$. \emph{The boundary of $X$
  in $\sigma(\ca)$} is the compact set
\begin{equation}\label{eq:cadag}
  \ca^\dagger=\sigma(\ca)\setminus X
  =\{\vkappa\in\sigma(\ca) \mid \vkappa(\varphi)=0
  \ \forall \varphi\in C_0(X) \}.  
\end{equation}
$\ca$ is canonically isomorphic with the $C^*$-algebra
$C(\sigma(\ca))$: each $\varphi\in\ca$ extends to a continuous
function on $\ca$ for which we keep the notation $\varphi$. We just
set $\varphi(\chi)=\chi(\varphi)$.

\subsection{Translations by characters}\label{ss:tchar}

The additive group $X$ naturally acts on $\sigma(\ca)$. Indeed,
$a\mapsto\tau_a$ is a strongly continuous homomorphism from $X$ into
the set of automorphisms of $\ca$ so if $\tau_a':\ca'\to\ca'$ is the
dual map then we get an action $a\mapsto\tau_a'$ of $X$ on the dual
space $\ca'$ of $\ca$.  If $\chi$ is a linear form on $\ca$ then
$\tau_a'(\chi)$ is the linear form $\tau_a'(\chi)=\chi\circ\tau_a$ so
if $\chi$ is a character then $\tau_a'(\chi)$ is a character. From
\eqref{eq:cadag} we see that this action of $X$ leaves invariant $X$
and $\ca^\dagger$ and $\tau_a'(\chi_b)=\chi_{a+b}$ for $a,b\in
X$. When there is no ambiguity we simplify the notation and set
$\tau_a'(\chi)=a+\chi$ for $a\in X$ and
$\chi\in\sigma(\ca)$.  \label{p:trans}

To each $\chi\in\sigma(\ca)$ we associate a morphism
$\tau_\chi:\ca\to C_\rmb^\rmu(X)$ uniquely determined by the condition
$\chi_x\circ\tau_\chi=\chi\circ\tau_x\forall x\in X$ and we say that
$\tau_\chi$ is the \emph{translation morphism associated to the
  character $\chi$}.  More explicitly, if $\varphi\in\ca$ then
$\tau_\chi(\varphi)$ is the function
\begin{equation}\label{eq:mtrans}
\tau_\chi(\varphi)(x)=\chi(\tau_x(\varphi)) \quad\forall x\in X.
\end{equation}
If $\chi=\chi_a$ is the evaluation at $a\in X$ then this is the usual
translation $\tau_a(\varphi)$. If $\chi=\vkappa\in\ca^\dagger$ then
think of $\tau_\vkappa(\varphi)$ as the translation of $\varphi\in\ca$
by the point at infinity $\vkappa$. 
Note that 
\[
\tau_\chi(\varphi)(x)=\chi(\tau_x(\varphi))
=(\tau_x'(\chi))(\varphi)=(x+\chi)(\varphi)
=\varphi(x+\chi).
\]
So $\tau_\chi(\varphi)$ is the function $x\mapsto \varphi(x+\chi)$
which we sometimes call \emph{localization of $\vphi$ at
  $\chi$}. \label{p:localiz}

The translations by characters $\chi\in\sigma(\ca)$ of functions
$\varphi\in\ca$ extend to translations of operators $A\in\ra$ as
follows: there is a unique continuous map $\tau_\chi:\ra\to\re$ such
that
\begin{equation}\label{eq:transm}
\tau_\chi(\varphi(q)\psi(p))=(\tau_\chi(\varphi))(q)\psi(p)
\quad \forall\varphi\in\ca\ \forall\psi\in C_0(X). 
\end{equation}
If $\chi=x\in X$ then $\tau_x(A)=\e^{\rmi xp}A\e^{-\rmi xp}$ and if
$\chi=\vkappa\in\ca^\dagger$ then
\begin{equation}\label{eq:tchar}
\tau_\vkappa(A)=\lim_{x\to\vkappa}\e^{\rmi xp}A\e^{-\rmi xp}
\end{equation}
where $x\in X$ tends to $\vkappa$ in $\sigma(\ca)$ and the limit holds
in the local norm topology on $\re$ (Theorem \ref{th:GIold}). We keep
the name \emph{translation morphism} for the map $\tau_\chi$ extended
to the operator level.

\subsection{Fredholm operators.}\label{ss:fred}

We recall some facts here concerning the essential spectrum of bounded
operators \cite[\S 4.3]{D0}.  Let $\ch$ be a Hilbert space.
$F\in B(\ch)$ is a \emph{Fredholm operator} if its kernel has finite
dimension and its range finite codimension (then the range is
closed). Equivalently, this means that there is an operator
$G\in B(\ch)$ such that $1-FG$ and $1-GF$ are compact. Clearly, this
can be rephrased as follows: \emph{the image of $F$ in the quotient
  $C^*$-algebra of $B(\ch)$ with respect to the ideal of compact
  operators $K(\ch)$ is an invertible operator} (Atkinson theorem
\cite[Th.\ 4.3.7]{D0}). It is this last version that we will use.

The spectrum of $A\in B(\ch)$ is denoted $\spec(A)$. Then the
\emph{essential spectrum} of $A$ is the set $\spe(A)$ of complex
numbers $\lambda$ such that $A-\lambda$ is not a Fredholm operator.

If $F\in\re$ is Fredholm then there is $G\in\rb$ such that
$1=FG+K$ with $K$ compact. Since $\|(\e^{\rmi xp}-1)F\|\to0$ and
$\|(\e^{\rmi xp}-1)K\|\to0$ as $x\to0$ we get
$\|(\e^{\rmi xp}-1)\|\to0$, which is impossible. Thus:

\begin{remark}\label{re:nfred} There are no Fredholm operators in
  $\re$. Thus $0\in\spe(A)$ if $A\in\re$.
\end{remark}

\subsection{Main results}\label{ss:mresults}

We recall a fact proved in \cite[Th.\ 5.16]{GI4} for $X$ an arbitrary
locally compact abelian groups.

\begin{theorem}\label{th:GIold}
  For any $A\in\ra$ the map
  $x\mapsto A_x\doteq\tau_x(A)= \e^{\rmi xp}A\e^{-\rmi xp}$ is norm
  continuous and extends to a continuous map
  $\sigma(\ca)\ni\chi\mapsto A_\chi\in\re_\loc$.  We have
  $A_\chi=\tau_\chi(A)$ where $\tau_\chi$ is the translation morphism
  associated to $\chi$. And
\begin{equation}\label{eq:kompact}
 \tau_\vkappa(A)=0 \ \forall\vkappa\in \ca^\dagger 
\Longleftrightarrow A\in\rk .
\end{equation}
\end{theorem}

A new proof of \eqref{eq:kompact} will be given in Section
\ref{s:proofold}. 

\begin{remark}\label{re:GIA}{\rm The morphism $\tau_\chi$, hence
    $A_\chi$, has been defined independently of the extension by
    continuity procedure used in Theorem \ref{th:GIold}. This is
    important in concrete situations because the computation of
    $A_\chi$ does not require a knowledge of the topology of
    $\sigma(\ca)$.  }\end{remark}

If $\rc$ is a $C^*$-algebra and $I$ is a set then we denote
$\prod_{i\in I}\rc$ the $C^*$-algebra consisting of all bounded
functions $C:I\to\rc$ with the natural operations and norm
$\|C\|=\sup_i\|C(i)\|$.  It follows that the map
$\Phi(A)=(A_\vkappa)_{\vkappa\in\ca^\dagger}$ defines a morphism
\begin{equation}\label{eq:Phi}
\Phi:\ra\to {\textstyle\prod_{\vkappa\in\ca^\dagger}}\re 
\end{equation}
whose kernel is $\rk$ hence it induces an injective morphism
\begin{equation}\label{eq:Phii}
\what\Phi:\ra/\rk\to {\textstyle\prod_{\vkappa\in
    \ca^\dagger}}\re .
\end{equation}
From \eqref{eq:Phii} we get for any \emph{normal} operator $A\in\ra$
(cf.\ \cite[Th.\ 1.15]{GI4}) 
\begin{equation}\label{eq:Phiii}
\spe (A)= \overline\ccup_{\vkappa\in\ca^\dagger} \spec(A_\vkappa)
\end{equation} 
where $\overline\bigcup$ means closure of the union.  Our main result
is an improvement of this relation. 

\begin{theorem}\label{th:GIA}
  For any operator $A\in\ra$ we have
\begin{equation}\label{eq:GIA}
\spe (A)= \ccup_{\vkappa\in\ca^\dagger} \spec(A_\vkappa).
\end{equation} 
\end{theorem}

The relation \eqref{eq:GIA} is a significative improvement of
\eqref{eq:Phiii}: the condition of normality of $A$ is eliminated and
one takes the union instead of the closure of the union. To get the
present version we use the techniques of M.~Lindner and M.~Seidel
\cite{LS} who solved a problem left open by V.~Rabinovich, S.~Roch,
and B.~Silberman, see \cite{RRS} and their earlier papers (one may
also find in \cite[\S1.4]{GI4} a detailed discussion of the earlier
literature). In their theory the Euclidean space $X$ is replaced by
the abelian group $\Z^d$ and instead of $\ra$ they work with algebras
similar to $\ell_\infty(X)\rtimes X$ acting in $\ell_p$ spaces of
Banach space valued functions.  Their results have been extended by
J.~\v{S}pakula and R.~Willett \cite{SW} (see also \cite{S}) to a
general class of discrete metric spaces (without any group structure)
under the condition that the space has the \emph{property A} in the
sense of Guoliang Yu. This property also plays a fundamental r\^ole in
\cite{G} where Theorem \ref{th:GIold} is extended to not necessarily
discrete metric spaces. \v{S}pakula and Willett also use the
\emph{metric sparsification property} of Chen, Tessera, Wang, and Yu
\cite{CTWY} and we follow them in this respect. Of course, in the
Euclidean case this could be replaced by an ad hoc construction, as in
\cite{LS}, but this simplifies a lot the argument and puts things in
the proper perspective. It seems clear to us that the proofs given in
Section \ref{s:proof} work for a class of (non-abelian) groups much
more general then the Euclidean spaces, e.g. locally compact groups
with finite asymptotic dimension, and this together with \cite[Th.\
6.8]{G} would give an analog of Theorem \ref{th:GIA} for such groups.
This would cover magnetic Schr\"odinger operators in the framework of
\cite[\S 5]{GI2}. We shall treat such extensions in a later
publication.  For an alternative approach to these topics, see the
paper \cite{NP} by V.\ Nistor and N.\ Prudhon.

In \cite{GI1,GI4} the space $X$ is an arbitrary locally compact
abelian group. Although the proofs in Section \ref{s:proof} clearly
extend to a more general class of non-abelian groups, we decided to
consider here only the case $X=\R^d$ which does not require much
formalism and the applications we have in mind concern only
differential operators on Euclidean spaces.

\begin{remark}\label{re:cbu}{\rm Theorem \ref{th:rb} is the particular
    case of Theorem \ref{th:GIA} corresponding to
    $\ca=C_\rmb^\rmu(X)$. The procedure of going from the characters
    of $C_\rmb^\rmu(X)$ to ultrafilters is explained in \S\ref{s:app}:
    for any $\chi\in\ca^\dagger$ there is $\vkappa\in X^\dagger$ such
    that $\chi(\varphi)=\lim_{x\to\vkappa}\varphi(x)$ for all
    $\varphi\in\ca$.  }\end{remark}

\begin{remark}\label{re:localize}{\rm The operator $A_\vkappa$ with
    $\vkappa\in\ca^\dagger$ will be called \emph{localization at
      $\vkappa$ of $A$} and we refer generically to the operators
    $A_\vkappa$ as \emph{localizations at infinity of $A$}. By
    \eqref{eq:transm} and a notation introduced page
    \pageref{p:localiz}, if $A$ is the pseudodifferential operator
    $\sum_i\varphi_i(q)\psi_i(p)$ then $A_\vkappa$ is the
    pseudodifferential operator $\sum_i\varphi_i(q+\vkappa)\psi_i(p)$.
    Theorem \ref{th:GIA} says that the essential spectrum of $A$ is
    determined by its localizations at infinity.  Note that this is
    not true in some simple and physically significative cases like
    the Stark Hamiltonian, cf.\ \S\ref{sss:stark}, and the constant
    magnetic field case. The point is that we consider only the
    ``infinity'' defined by the position observable $q$, while for
    other Hamiltonians one has to take into consideration the
    contribution of other regions at infinity in phase space.
  }\end{remark}

\begin{remark} {\rm If $\varphi\in C^\rmu_\rmb(X)$ and
    $\psi\in C_0(X)$ then the operator $A=\varphi(q)\psi(p)$ belongs
    to $\re$ hence it has localizations at infinity. On the other
    hand, if $\varphi(x)=\e^{\rmi x^2}$ and
    $0\neq\cf\psi\in C^\infty_\rmc(X)$ then $A\nin\re$ \cite[Ex.\
    7.2]{G}. The importance of the uniform continuity condition can
    also be seen as follows: it is clear that $A$ is an integral
    operator with kernel of the form $k(x,y)=\e^{\rmi x^2}\theta(x-y)$
    where $\theta\in C^\infty_\rmc(X)$, so $k$ is controlled bounded
    and $C^\infty$ but not uniformly continuous. Moreover, the
    operator $A$ has no localizations at infinity. Indeed, we have
    $A_x=\e^{\rmi xp}A\e^{-\rmi xp}=\e^{\rmi(x^2+2qx)}A$ hence
    $\|A_xu\|=\|Au\|\neq0$ if $u\neq0$ and clearly
    $\wlim_{x\to\infty}A_xu=0$. Thus if $u\neq0$ then a sequence of
    vectors of the form $A_{x_n}u$ with $x_n\to\infty$ cannot converge
    strongly.}
\end{remark}

\begin{remark}\label{re:restrict}{\rm The character space
    $\ca^\dagger$ may sometimes be reduced to a much smaller set by
    the following procedure. It may happen that there is
    $\Omega\subset\ca^\dagger$ such that for each
    $\vkappa\in\ca^\dagger$ the morphism $\tau_\vkappa$ may be
    factorized $\tau_\vkappa=\eta\tau_\omega$ with $\omega\in\Omega$
    and $\eta$ a morphism. If this is the case then one has
    $\spec(A_\vkappa) =\spec(\eta(\tau_\omega(A))) \subset
    \spec(\tau_\omega(A))=\spec(A_\omega)$
    hence \eqref{eq:GIA} will hold with $\ca^\dagger$ replaced by
    $\Omega$. We will see an example of this mechanism in
    \S\ref{ss:hom}.  }\end{remark}

\section{Proof of Theorem \ref{th:GIold}.}\label{s:proofold}
\protect\setcounter{equation}{0}

We shall give here a proof of \eqref{eq:kompact} which requires much
less formalism and is simpler than that from \cite{GI4}. Clearly it
suffices to consider the case $\ca=C_\rmb^\rmu(X)$.

In this section we abbreviate $\ce=C_\rmb^\rmu(X)$.  We denote $\rx$
the uniform compactification of $X$, i.e. the character space
$\sigma(\ce)$ of the algebra $\ce$, and let
$\cx=\rx\setminus X=\ce^\dagger$ be the boundary of $X$ in $\rx$.
Then, according to the first part of Theorem \ref{th:GIold}, for each
$A\in\re$ there is a continuous map
$\rx\ni\chi\mapsto A_\chi\in\re_\loc$ such that
$A_x=\e^{\rmi xp}A\e^{-\rmi xp}$ if $x\in X$. We will show that
$A\in\rk$ if $A_\chi=0$ for all $\chi\in\cx$.

Let us state this in more explicit terms. Let $\theta\in C_0(X)$ be a
strictly positive function. Then $\chi\mapsto A_\chi\theta(q)$ is a
norm continuous function $\rx\to\rb$ and, since $\rx$ is compact, it
is uniformly continuous. We assume that this function is zero on the
boundary of $X$ in its compactification $\rx$. This means in fact that
the restriction to $X$ of the function $\chi\mapsto A_\chi\theta(q)$
belongs to the space $C_0(X,\rb)$ of continuous functions $X\to\rb$
which are of class $C_0$. Even more explicitly, this means that
\[
x\mapsto A_x\theta(q)=\e^{\rmi xp}A\e^{-\rmi xp}\theta(q)
\]
is a norm continuous function $X\to\rb$ which tends to zero at
infinity. Or
\[
\lim_{x\to\infty}\|A\e^{-\rmi xp}\theta(q)\|=0. 
\]
Of course, this relation will hold for any $\theta\in C_0(X)$: first
we get it for $\theta$ replaced by any continuous function with
compact support and then we extend it to any function in $C_0(X)$ by
density and uniform boundedness of $A\e^{-\rmi xp}$. Since
$\e^{-\rmi xp}\theta(q)\e^{\rmi xp}=\theta(q-x)$ we get
\begin{equation}\label{eq:zero}
\lim_{x\to\infty}\|A\theta(q-x)\|=0 \quad \forall \theta\in C_0(X). 
\end{equation}
To summarize, we have to prove the following: if $A\in\re$, which
means 
\begin{equation}\label{eq:GIchar}
  \lim_{a\to0}\|(\e^{\rmi ap}-1)A^{(*)}\|=0
  \text{ and } \lim_{a\to0}\|\e^{-\rmi aq}A\e^{\rmi aq}-A\|=0
\end{equation}
and if \eqref{eq:zero} is satisfied, then $A\in\rk$. By using the
Riesz-Kolmogorov characterization of compactness, and by taking into
account that $A$ is compact if and only if $A^*$ is compact, we see
that it suffices to prove 
\[
\lim_{a\to0}\|A(\e^{\rmi ap}-1)\|=0
  \text{ and } \lim_{a\to0}\|A(\e^{\rmi aq}-1)\|=0.
\] 
But the first of these relations is automatically satisfied because of
\eqref{eq:GIchar} hence in fact we just have to prove that
$\lim_{a\to0}\|A(\e^{\rmi aq}-1)\|=0$ if \eqref{eq:zero} and
\eqref{eq:GIchar} are satisfied.

\begin{lemma}\label{lm:limT}
  Let $\phi_r$ be the characteristic function of the region
  $|x|>r$. Then for any $A\in\rb$ we have
\begin{equation}\label{eq:limT}
\lim_{a\to0}\|A(\e^{\rmi aq}-1)\|=0  \Longleftrightarrow
\lim_{r\to\infty}\|A\phi_r(q)\|=0
\end{equation}
\end{lemma}

\begin{proof}
  From \cite[Cor.\ 3.3]{GI30} with $T=A^*$ it follows that
  $\lim_{a\to0}\|A(\e^{\rmi aq}-1)\|=0$ if and only if
  $\forall\varepsilon>0\ \exists\theta\in C_\rmc(X)$ such that
  $\|A\theta^\perp(q)\|<\veps$; here $\theta^\perp=1-\theta$. If
  $\psi=|\theta^\perp|^2$ then
\[
\|A\theta^\perp(q)\|^2=\|A\psi(q)A^*\|\geq
\|A\phi_r(q)A^*\|=\|A\phi_r\|^2
\]
if $\supp\theta\subset B_0(r)$. This proves $\Rightarrow$ in
\eqref{eq:limT}.  Reciprocally, if $\|A\phi_r(q)\|<\varepsilon$ then
choose $\theta$ continuous such that $\theta(x)=1$ for $|x|<r$,
$\theta(x)=0$ for $|x|>2r$, and $0\leq\theta\leq1$.
\end{proof}

Let $B_x(r)=\{y\in X\mid |y-x|<r\}$ and $B_x=B_x(1)$. We denote
$1_{B_x}$ the characteristic function of $B_x$ and often identify
$1_{B_x}(q)=1_{B_x}$.

\begin{lemma}\label{lm:limB}
For any $A\in\rb$ we have
\begin{equation}\label{eq:limB}
  \lim_{x\to\infty}\|A1_{B_x}\|=0 \Longleftrightarrow
  \lim_{x\to\infty}\|A\theta(q-x)\|=0\ \forall\theta\in C_0(X).
\end{equation}
\end{lemma}

\begin{proof}
  If $B=B_0$ then $0\leq 1_{B_x}(q)=1_B(q-x)\leq \theta(q-x)$ if
  $\theta\in C_0(X)$ and $\theta\geq1$ on the unit ball, hence the
  implication $\Leftarrow$ is obvious. Reciprocally, it suffices to
  show that the left hand side of \eqref{eq:limB} implies the right
  hand side if $\theta$ has compact support. Then if $Z$ is a finite
  set such that $\supp\theta\subset\cup_{z\in Z}B_z$ then there is a
  number $C$ such that
\[
|\theta(q-x)|^2\leq C\sum_{z\in Z}1_{B_z}(q-x)
\leq C\sum_{z\in Z}1_{B_{x+z}}(q)
\]
hence
\begin{align*}
\|A\theta(q-x)\|^2 &=\|A|\theta(q-x)|^2A^*\|
\leq C\sum_{z\in Z}\|A1_{B_{x+z}}(q)A^*\| \\
&=C\sum_{z\in Z}\|A1_{B_{x+z}}(q)\|^2
\end{align*}
and the right hand side tends to zero as $x\to\infty$. 
\end{proof}

\begin{lemma}\label{lm:niet}
  There is a family of linear maps $\Phi_\varepsilon:\re\to\re$, with
  $\varepsilon>0$, such that for any $A\in\re$ the operator
  $A_\varepsilon\doteq\Phi_\varepsilon(A)$ is controlled,
  $\|A_\varepsilon\|\leq \|A\|$, and
  $\lim_{\varepsilon\to0}A_\varepsilon=A$ in norm. Moreover,
  $\Phi_\varepsilon(A\varphi(q))=\Phi_\varepsilon(A)\varphi(q)$ for
  all $\varphi\in C_0(X)$.
\end{lemma}

\begin{proof}
The next argument is based on an idea from the proof of Proposition
4.11 from \cite{GG}. To each $\xi\in L^1(X)$ real we associate a map
$\Phi_\xi:\rb\to\rb$ defined by
\[
\Phi_\xi(A)=\int_X \e^{-\rmi kq} A \e^{\rmi kq} \xi(k) \d k . 
\]
Clearly $\Phi_\xi$ is linear and norm continuous, in fact
$\|\Phi_\xi(A)\|\leq \|\xi\|_{L^1}\|A\|$. It is also clear that
$\Phi_\xi(A\varphi(q))=\Phi_\xi(A)\varphi(q)$ and
$\Phi_\xi(\varphi(q)A)=\varphi(q)\Phi_\xi(A)$ if
$\varphi\in L^\infty(X)$.

Let us write $A\in C^\rmu(q)$ if the second condition in
\eqref{eq:GIchar} is satisfied. Since
$\e^{-\rmi aq}\Phi_\xi(A)\e^{\rmi aq} = \Phi_\xi(\e^{-\rmi
  aq}A\e^{\rmi aq})$
we obviously have $A\in C^\rmu(q)\Rightarrow\Phi_x(A)\in C^\rmu(q)$.
Then by using the relation
$\e^{\rmi ap}\e^{-\rmi kq}=\e^{-\rmi ka}\e^{-\rmi kq}\e^{\rmi ap}$ we
get
\begin{align*}
(\e^{\rmi ap}-1)\Phi_\xi(A)
&= \int_X \e^{-\rmi kq} (\e^{\rmi ap}-1)A \e^{\rmi kq} \xi(k) \d k \\
&+\int_X (\e^{-\rmi ka}-1)\e^{-\rmi kq}\e^{\rmi ap}A \e^{-\rmi kq} \xi(k) \d k
\end{align*}
so that
\[
\|(\e^{\rmi ap}-1)\Phi_\xi(A)\| \leq
\|(\e^{\rmi ap}-1)A\| \|\xi\|_{L^1}
+\|A\|\int_X |\e^{-\rmi ka}-1| \, |\xi(k)| \d k. 
\]
Note also that $\Phi_\xi(A)^*=\Phi_\xi(A^*)$. Thus we see that if $A$
satisfies the first condition in \eqref{eq:GIchar} then $\Phi_\xi(A)$
satisfies it too. So, we proved that \emph{$\Phi_\xi$ leaves $\re$
  invariant}.

Now we prove that \emph{if the Fourier transform $\hat\xi$ of $\xi$
  has compact support then $\Phi_\xi(A)$ is a controlled operator, for
  any $A\in\rb$}. The next computation is slightly formal but easy to
justify \cite[Prop.\ 4.11]{GG}. If $\varphi,\psi$ are bounded
continuous functions then $\varphi(q)=\int_X \varphi(x) E(\d x)$ where
$E$ is the spectral measure of the observable $q$ and similarly for
$\psi$. Then 
\begin{align*}
\varphi(q) & \Phi_\xi(A)\psi(q) =\int_{X} 
\varphi(q)\e^{-\rmi kq}A\e^{\rmi kq}\psi(q)\xi(k)\d k \\
&=\int_{X^3} 
\varphi(x)\e^{-\rmi kx} E(\d x)AE(\d y)\e^{\rmi ky}\psi(y)\xi(k)\d k
  \\
&=\int_{X^2} \varphi(x) \psi(y) \hat\xi(x-y) E(\d x)AE(\d y) .
\end{align*}
So if $\hat\xi(z)=0$ for $|z|>R$ and if the distance between the
supports of $\vphi$ and $\psi$ is $>R$ then
$\varphi(q)\Phi_\xi(A)\psi(q) =0$, hence  $\Phi_\xi(A)$ is a
controlled operator. 

Let us fix a positive function $\xi$ such that $\int_X\xi(k)\d k=1$
which is the Fourier transform of a function with compact support.
For $\varepsilon>0$ we set
$\xi_\varepsilon(k)=\varepsilon^{-d}\xi(k/\varepsilon)$, function
whose Fourier transform is $x\mapsto\hat\xi(\varepsilon x)$ which is
also of compact support. If we set
$\Phi_\varepsilon=\Phi_{\xi_\varepsilon}$ then
\[
\Phi_\varepsilon(A)=
\int_X \e^{-\rmi kq} A \e^{\rmi kq} \xi_\varepsilon(k) \d k 
=\int_X \e^{-\rmi\varepsilon kq} A \e^{\rmi\varepsilon kq} \xi(k) \d k 
\] 
hence clearly $\lim_{\varepsilon\to0}\Phi_\varepsilon(A)=A$ in norm if
$A\in C^\rmu(q)$.
\end{proof}

\begin{proposition}\label{pr:kompact}
  If $A\in\re$ then $A\in\rk$ if and only if
  $\lim_{x\to\infty}\|A1_{B_x}\|=0$.
\end{proposition}

\begin{proof}
  If $A\in\rk$ then $\lim_{x\to\infty}\|A1_{B_x}\|=0$ is true if $A$
  has rank $1$, hence if $A$ is compact. Reciprocally, let $A\in\re$
  such that $\lim_{x\to\infty}\|A1_{B_x}\|=0$ and let $A_\varepsilon$
  be as in Lemma \ref{lm:niet}. Since $\|A_\varepsilon-A\|\to0$ as
  $\varepsilon\to0$ it suffices to show that $A_\varepsilon$ is
  compact. The operator $A_\varepsilon$ belongs to $\re$ and is
  controlled. Moreover, if $\theta\in C_0(X)$ then
  $A_\varepsilon\theta(q-x)=\Phi_\varepsilon(A\theta(q-x))$ hence
  $\|A_\varepsilon\theta(q-x)\|\leq\|A\theta(q-x)\|$ by Lemma
  \ref{lm:niet} and so
  $\lim_{x\to\infty}\|A_\varepsilon\theta(q-x)\|=0$ by Lemma
  \ref{lm:limB} hence $\lim_{x\to\infty}\|A_\varepsilon1_{B_x}\|=0$ by
  the same lemma. Since $A_\varepsilon\in\re$ the operators
  $\theta(q)A_\varepsilon$ and $A_\varepsilon\theta(q)$ are compact if
  $\theta\in C_\rmc(X)$. Thus, due to \cite[Lem.\ 3.8]{G},
  $A_\varepsilon$ is a compact operator. The quoted lemma says that
  $\lim_{r\to\infty}\|A_\varepsilon\phi_r\|=0$ holds because
  $\lim_{x\to\infty}\|A_\varepsilon1_{B_x}\|=0$ and $A_\varepsilon$ is
  controlled. Then we get compactness by Lemma \ref{lm:limT}.
\end{proof}

\section{Proof of Theorem \ref{th:GIA}.}\label{s:proof}
\protect\setcounter{equation}{0}

The algebra $\ra$ has no unit and its unitization may be identified
with the $C^*$-subalgebra $\ra_1=\ra+\C\subset\rb$. Then the map
$\Phi$ introduced in \eqref{eq:Phi} extends to a unital morphism
$\ra_1\to\prod_{\vkappa\in\ca^\dagger}\rb$, that we also denote
$\Phi$, and the kernel of this extension is $\rk$.  If
$S=A-\lambda$ with $A\in\ra$ and $\lambda\in\C$ then
$S_x=\e^{\rmi xp}S\e^{-\rmi xp}=\e^{\rmi xp}A\e^{-\rmi xp} -\lambda$
hence the map $x\mapsto S_x\in\rb$ extends to a continuous map
$\sigma(\ca)\ni\chi\mapsto S_\chi\in\rb_\loc$ and we have
$S_\chi=A_\chi-\lambda$ for all $\chi\in\sigma(\ca)$. Thus the
extended $\Phi$ is given by the same formula
$\Phi(S)=(S_\vkappa)_{\vkappa\in\ca^\dagger}$ for $S\in\ra_1$.

Let $\what{S}$ be the quotient of $S$ in $\rb/\rk$. Then $S$ is
Fredholm if and only if $\what{S}$ is invertible in
$\rb/\rk$. If $S\in\ra_1$ then this happens if and only if
$\Phi(S)$ is invertible in $\prod_{\vkappa\in\ca^\dagger}\rb$, hence 
\begin{equation}\label{eq:fredholmmm}
  S\in\ra_1 \text{ is Fredholm} 
  \Leftrightarrow
\begin{cases} 
S_\vkappa \text{ is invertible  in } \rb
  \ \forall\vkappa\in\ca^\dagger \\
\text{ and } 
{\textstyle\sup_{\vkappa\in\ca^\dagger}}\|S_\vkappa^{-1}\|<\infty.
\end{cases}
\end{equation}
On the other hand, the relation \eqref{eq:GIA} is equivalent to
\[
\C\setminus\spe(A)= {\textstyle\bigcap_{\vkappa\in\ca^\dagger}}
\big(\C\setminus\spec(A_\vkappa)\big)
\]
which means that we have $\lambda\nin\spe(A)$
if and only if $A_\vkappa-\lambda$ is an invertible operator in
$\rb$ for all $\vkappa\in\ca^\dagger$.  Thus, with the notation
$S=A-\lambda$, \eqref{eq:GIA} says
\begin{equation}\label{eq:fredholmm}
  S \text{ is Fredholm }
  \Leftrightarrow
  S_\vkappa \text{ is invertible  in } \rb
  \ \forall\vkappa\in\ca^\dagger. 
\end{equation}
So to prove the theorem we have to show that the second condition
\eqref{eq:fredholmmm} is automatically satisfied. To summarize, we
have to prove
\begin{equation}\label{eq:show}
S\in\ra_1 \text{ and }S_\vkappa\text{ invertible } 
\ \forall\vkappa\in\ca^\dagger \Longrightarrow
{\textstyle\sup_{\vkappa\in\ca^\dagger}}\|S_\vkappa^{-1}\|<\infty.
\end{equation}
Following \cite{LS} we set $\nu(T)\doteq\inf_{\|u\|=1}\|Tu\|$ if $T$
is a bounded operator in a Hilbert space. If $T$ is bijective then
$\nu(T)=\|T^{-1}\|^{-1}$.  Thus the relation
$\sup_{\vkappa\in\ca^\dagger}\|S_\vkappa^{-1}\|<\infty$ is equivalent
to $\inf_{\vkappa\in\ca^\dagger}\nu(S_\vkappa)>0$. Hence
\eqref{eq:show} follows from
\begin{equation}\label{eq:main}
\forall S\in\ra_1 \ \exists\,\omega\in\ca^\dagger \text{ such that }
{\textstyle\inf_{\vkappa\in\ca^\dagger}}\,\nu(S_\vkappa)=\nu(S_\omega). 
\end{equation}
because in our case $\what{S}$ is invertible, so we know that
$S_\vkappa$ is invertible $\forall\vkappa\in\ca^\dagger$.  We will
prove \eqref{eq:main} by adapting the ideas of the papers \cite{LS,SW} to the
Euclidean context.

If $L\in\rb$ then localized versions of $\nu(L)$ are defined as
follows. Let $\Omega\subset X$ open and $0<\theta\leq\infty$. We
assume (here and later) that $\Omega$ is not empty, so $L^2(\Omega)$
is a closed not trivial subspace of $L^2(X)$, and we set
\begin{align*}
\nu(L|\Omega) &=\inf\{\|Lu\|\mid u\in L^2(\Omega),\ \|u\|=1\}, \\
\nu_\theta(L|\Omega) &=\inf\{\|Lu\| \mid u\in L^2(\Omega), 
\ \diam(\supp u)<\theta,\ \|u\|=1\} .
\end{align*}
We have  $\nu(L|\Omega)=\nu_\infty(L|\Omega)$ because the functions
with compact support are dense in $L^2(\Omega)$. Moreover, one may
easily check the relations:
\begin{align}
&\nu_{\theta'}(L|\Omega') \leq \nu_{\theta''}(L|\Omega'')
\quad\text{if } \theta'\geq\theta'' , \Omega'\supset\Omega'',
\label{eq:ineq} \\
&\nu_\theta(\e^{\rmi ap}L\e^{-\rmi ap}|\Omega)=
\nu_\theta(L|a+\Omega) \quad \forall a\in X .
\label{eq:aOmega}
\end{align}
Let us denote $\rho(L)$ the lower bound of the numbers $r$ such that
$\braket{u}{Lv}=0$ if the distance between the supports of $u$ and $v$
is $>r$.  Then $L$ is called \emph{controlled} if $\rho(L)<\infty$.
This means that there is $r>0$ such that $\braket{u}{Lv}=0$ whenever
the distance between the supports of $u$ and $v$ is larger than $r$
\cite[pp. 67, 69]{R}.

The proofs of the next two lemmas closely follow those of Proposition
7.6 and Corollary 7.10 from \cite{SW}.

\begin{lemma}\label{lm:nuc}
  For any $\varepsilon>0,r>0,\ell<\infty$ there is $\theta>0$ such that
  $\forall\, \Omega\subset X$ open
\[
\nu_\theta (L|\Omega)\leq\nu(L|\Omega)+\varepsilon \quad\text{if } 
L\in\rb \text{ with } \|L\|\leq \ell \text{ and } \rho(L)\leq r.
\]
\end{lemma}

\begin{proof}
  $X=\R^d$ has the metric sparsification property, for example because
  it has finite asymptotic dimension \cite{CTWY}; see also
  \cite{NY,R}. To state this in precise terms we introduce a notation
  and a notion. We denote $\cm(X)$ the set of finite positive measure
  on $X$. Then let $R>0$ a real number; a closed subset $Y\subset X$
  will be called \emph{$R$-sparse} if it can be written as a union
  $Y=\ccup_{i\in I}Y_i$ with $\dist(Y_i,Y_j)\geq R$ if $i\neq j$ and
  $\sup_i\diam(Y_i)<\infty$. Clearly each $Y_i$ will also closed.
  Then, by taking into account \cite[Prop. 3.3]{CTWY}, we have:
\begin{equation}\label{eq:spar}
c<1,\, R> 0,\, \mu\in\cm(X) \Rightarrow \exists
Y=R\text{-sparse set } \text{ with } \mu(Y)\geq c\mu(X) .
\end{equation}
In the Definition 3.1 from \cite{CTWY} the set $Y$ is only assumed
Borel but it is easy to check that its closure will have the same
properties.  We choose $c,R$ such that
\begin{equation}\label{eq:cR}
1/2<c<1,\ 6\ell(c^{-1}-1)^{1/2}< \varepsilon, \ R>2r .
\end{equation}
Now let $u\in L^2(\Omega)$ with $\|u\|=1$ and
$\|Lu\|<\nu(L|\Omega)+\varepsilon/4$. Define the measure $\mu$ by
$\mu(A)=\int_A|u(x)|^2\d x$. Finally, let $Y$ as in \eqref{eq:spar}
and choose any $\theta>\sup_i\diam(Y_i)$.  Denote $1_Y$ the
characteristic function of $Y$ and $1'_Y=1-1_Y$ the characteristic
function of $X\setminus Y$. If $u_Y=1_Yu$ then
\[
\|Lu-Lu_Y\|^2\leq \|L\|^2\|1'_Yu\|^2 =\|L\|^2\mu(X\setminus Y)
\leq \|L\|^2 (1-c)\mu(X) = (1-c)\|L\|^2
\]
hence $\|Lu_Y\|\leq\|Lu\|+\ell(1-c)^{1/2}$.  Set $u_i=1_{Y_i}u$, so
$Lu_Y=\sum_i Lu_i$ and the functions $Lu_i$ have disjoint supports,
hence $\|Lu_Y\|^2=\sum_i \|Lu_i\|^2$. We keep only the terms with
$u_i\neq0$ in this sum and write
\begin{align*}
\|Lu_Y\|^2 = {\textstyle\sum_i}\|Lu_i\|^2 /\|u_i\|^2 \cdot \|u_i\|^2 
\geq {\textstyle\inf_i}\, \|Lu_i\|^2 /\|u_i\|^2 
{\textstyle\sum_j \|u_j\|^2}
\end{align*}
and since we also have $\sum_j\|u_j\|^2=\|u_Y\|^2\geq c\|u\|^2=c$ we
get
\begin{align*}
{\textstyle\inf_i}\, \|Lu_i\| /\|u_i\|
&\leq c^{-1/2}\|Lu\| + \ell(c^{-1}-1)^{1/2}  \\
&\leq \|Lu\| +\ell\big(c^{-1/2}-1 +(c^{-1}-1)^{1/2} \big) .
\end{align*}
If $a=c^{-1}$ then from \eqref{eq:cR} we get $1 <a<2$ and
\begin{align*}
a-1 &+(a^2-1)^{1/2} =(a-1)^{1/2}\big((a-1)^{1/2}+(a+1)^{1/2}\big) \\
& <(a-1)^{1/2}(1+\sqrt3) <3(a-1)^{1/2} < \varepsilon/2\ell .
\end{align*}
Thus $\inf_i \|Lu_i\| /\|u_i\| \leq \|Lu\| +\varepsilon/2$.  Choose
$i$ such that $\|Lu_i\| /\|u_i\| \leq \|Lu\| +3\varepsilon/4$ and
denote $v=u_i/\|u_i\|\in L^2(\Omega)$. Then $\supp v\subset Y_i$ which
has diameter $<\theta$ and we have
$\|Lv\|\leq \nu(L|\Omega)+\varepsilon$ by the choice of $u$. Hence
$\nu_\theta (L|\Omega)\leq\nu(L|\Omega)+\varepsilon$.
\end{proof}

\begin{lemma}\label{lm:nuca}
  If $S\in\ra_1$ and $\varepsilon>0$ then there is $\theta>0$ such
  that
\[
\nu_\theta (S_\vkappa|\Omega)\leq\nu(S_\vkappa|\Omega)+\varepsilon \quad
\forall\vkappa\in\ca^\dagger \text{ and } \forall\, \Omega\subset X 
\text{ open} .
\]
\end{lemma}

\begin{proof}
  We have $S=A-\lambda$ with $A\in\ra$ and $\lambda\in\rc$. The subset
  of $\ra$ consisting of controlled operators is dense in $\ra$:
  indeed, it contains the linear subspace generated by the operators
  $\varphi(q)\psi(p)$ with $\varphi\in\ca$ and $\cf\psi\in C_\rmc(X)$.
  So there is a controlled operator $T\in\ra+\C$ such that
  $\|S-T\|<\veps$. This clearly implies
  $\|S_\vkappa-T_\vkappa\|<\veps$ for all
  $\vkappa\in\ca^\dagger$. Moreover, $\rho(T_\vkappa)\leq\rho(T)$
  because if $r>\rho(T)$ and $\dist(\supp{u},\supp{v})>r$ then for
  each $x\in X$ we also have
  $\dist(\supp(\e^{\rmi xp}u),\supp(\e^{\rmi xp}v))>r$ hence
  $\braket{u}{T_xv}=0$ so by passing to the limit in the direction
  $\vkappa$ we get $\braket{u}{T_\vkappa v}=0$, so
  $\rho(T_\vkappa)\leq r$. Since we also have
  $\|T_\vkappa\|\leq\|T\|\leq\|S\|+\varepsilon $, we may use Lemma
  \ref{lm:nuc} and find $\theta>0$ such that
\[
\nu_\theta (T_\vkappa|\Omega)\leq\nu(T_\vkappa|\Omega)+\varepsilon \quad
\forall\vkappa\in\ca^\dagger \text{ and } \forall\,\Omega\subset X.
\]
Finally, from $\|S_\vkappa-T_\vkappa\|<\veps$ we obviously get 
$|\nu_\theta (S_\vkappa|\Omega)-\nu_\theta(T_\vkappa|\Omega)|\leq\veps$ and
similarly for the $\nu(\cdot\,|\Omega)$ quantities. 
Thus $\nu_\theta(S_\vkappa|\Omega)\leq\nu(S_\vkappa|\Omega)+3\varepsilon$
$\forall\vkappa\in\ca^\dagger$ { and } $\forall\,\Omega\subset X$.
\end{proof}

In the proof of the next lemma we use an argument from the proof of
Theorem 8 in \cite{LS}. Let $B(r)=\{x\in X\mid |x|<r\}$ be the open
ball of center $0$ and radius $r$ in $X$. If $S\in\rb$ then the
operators $S_a=\e^{\rmi ap}S\e^{-\rmi ap}$ are called
\emph{translations of $S$}.

\begin{lemma}\label{lm:imppp}
  Let $S\in\rb$ and $n>0$ integer. For $1\leq i\leq n$ let
  $\varepsilon_i,\theta_i>0$ such that
\begin{equation}\label{eq:llsii}
  \nu_{\theta_i}(S|\Omega) < \nu(S|\Omega)+\varepsilon_i 
  \quad \forall\, \Omega\subset X \text{ open ball}.
\end{equation}
Then there is a translate $T$ of $S$ such that for all $1\leq m\leq n$
\begin{equation}\label{eq:imppp}
\nu(T|B({\textstyle\sum_{1\leq i\leq m}}\theta_i+\theta_{m})) 
<\nu(S) +\varepsilon_m+\dots+\varepsilon_{n} .
\end{equation}
\end{lemma}

\begin{proof}
  To simplify the writing, in this proof we set $V_a=\e^{\rmi ap}$. We
  begin with two remarks which will be used in the next
  argument. First, due to \eqref{eq:aOmega}, the estimate
  \eqref{eq:llsii} is also satisfied by any translate of $S$. Then, if
  $v$ is a function with $\diam(\supp{v})<\theta$ and if
  $a\in\supp{v}$ then $\supp{V_av}\subset B(\theta)$ because
  $\supp{V_av}=\supp{v}-a$.

  Denote $(\varepsilon_i')$ a permutation of the numbers
  $(\varepsilon_i)$ (we shall specify it later on) and let
  $(\theta_i')$ be the corresponding numbers defined as in
  \eqref{eq:llsii}. Set $S_0=S$ and $\theta'_0=\infty$.  Let us prove
  that there are translations $S_1,\dots,S_n$ of $S$ such that
\begin{equation}\label{eq:svx}
S_{i}=V_{x_i}S_{i-1}V_{x_i}^* \text{ with } |x_i|<\theta'_{i-1}
\quad\text{and}\quad \nu(S_i|B(\theta'_i))
<\nu(S)+\varepsilon'_{1}+\dots+\varepsilon'_{i} . 
\end{equation}
From \eqref{eq:llsii} we get
$\nu_{\theta'_1}(S) < \nu(S)+\varepsilon'_1$ hence there is
$v\in L^2(X)$ with $\|v\|=1$ and $\diam(\supp{v})<\theta'_1$ such that
$\|Sv\|<\nu(S)+\varepsilon'_1$.  If $x_1\in\supp{v}$ and $u_1=V_{x_1}v$
then $\supp u_1\subset B(\theta'_1)$ so if we set
$S_1=V_{x_1}SV_{x_1}^*$ then
\[
\|S_1u_1\|<\nu(S)+\varepsilon'_1
\quad\text{hence}\quad \nu(S_1|B(\theta'_1)) 
< \nu(S)+\varepsilon'_1.
\]
Thus $S_1$ has been constructed. Assume that $S_1,\dots,S_i$ have been
constructed with $i<n$ and let us construct $S_{i+1}$. From
\eqref{eq:llsii} and the induction assumption \eqref{eq:svx} we get
\[
  \nu_{\theta'_{i+1}}(S_i|B(\theta'_i))
< \nu(S_i|B(\theta'_i))+\varepsilon'_{i+1}
< \nu(S)+\varepsilon'_{1}+\dots+\varepsilon'_{i}+\varepsilon'_{i+1} .
\]
So there is unit vector $v\in L^2(X)$ with
$\supp v\subset B(\theta'_i)$ and $\diam(\supp v)<\theta'_{i+1}$ such
that $\|S_{i}v\|< \nu(S)+\varepsilon'_1+\dots+\varepsilon'_{i+1}$. Let
$x_{i+1}\in\supp{v}$ and let us denote $u_{i+1}=V_{x_{i+1}}v$ and
$S_{i+1}=V_{i+1}S_iV_{i+1}^*$. Then $|x_{i+1}|<\theta'_i$,
$\supp{u_{i+1}}\subset B(\theta'_{i+1})$, and
\[
\|S_{i+1} u_{i+1}\|<\nu(S)+\varepsilon'_1+\dots+\varepsilon'_{i+1} 
\quad\text{hence}\quad
\nu(S_{i+1}|B(\theta'_{i+1})) <
\nu(S)+\varepsilon'_1+\dots+\varepsilon'_{i+1}.  
\]
This proves the existence of operators $S_1,\dots,S_n$ verifying
\eqref{eq:svx}. 

Next, starting with $S_n=V_{x_n}S_{n-1}V_{x_n}^*$ then replacing
$S_{n-1}$ by $V_{x_{n-1}}S_{n-2}V_{x_{n-1}}^*$ and so on, we get
$S_n=V_{y_i}S_iV_{y_i}^*$ with $y_i=x_n+\dots+x_{i+1}$.  Then in
\eqref{eq:svx} we use $S_i=V_{y_i}^*S_nV_{y_i}$ and by taking into
account the relation \eqref{eq:aOmega} we get 
\[
\nu(S_n|B(\theta'_i)-y_i)=\nu(V_{y_i}^*S_nV_{y_i}|B(\theta'_i))
=\nu(S_i|B(\theta'_i))<\nu(S)+\varepsilon'_{1}+\dots+\varepsilon'_{i}.
\]
Since $|y_i|\leq|x_n|+\dots+|x_{i+1}|<\theta'_i+\dots+\theta'_{n-1}$ we
have
\[
B(\theta'_i)-y_i\subset B(\theta'_i+|y_i|)
\subset B(2\theta'_i+\dots+\theta'_{n-1})
\subset B(2\theta'_i+\dots+\theta'_{n})
\]
hence 
\[
\nu(S_n|B(2\theta'_i+\dots+\theta'_{n}))
<\nu(S)+\varepsilon'_{1}+\dots+\varepsilon'_{i} .
\]
If we take $\varepsilon_{k}'=\varepsilon_{n-k+1}$ and
$\theta_{k}'=\theta_{n-k+1}$ we get
\[
\nu(S_n|B(2\theta_{n-i+1}+\dots+\theta_{1}))
<\nu(S)+\varepsilon_{n}+\dots+\varepsilon_{n-i+1}
\]
hence \eqref{eq:imppp} is satisfied with $T=S_n$. 
\end{proof}

\begin{lemma}\label{lm:topology}
  For each $S\in\ra_1$ the set
  $\cs=\{S_\vkappa\mid \vkappa\in\ca^\dagger\}$ is a compact stable
  under translations subset of $\rb_\loc$.
\end{lemma}

\begin{proof}
  As explained at the beginning of Section \ref{s:proof}, the map
  $\chi\mapsto S_\chi$ is a continuous function
  $\sigma(\ca)\to\rb_\loc$. Since $\ca^\dagger$ is a compact subset
  of $\sigma(\ca)$ it follows that the set $\cs$ is compact in
  $\rb_\loc$. To prove the invariance under translations, it
  suffices to note that if $a\in X$ and $\vkappa\in\ca^\dagger$ then
  $\tau_a(S_\vkappa)=\tau_a\tau_\vkappa(S) =\tau_{a+\vkappa}(S)$ and
  $a+\vkappa\in\ca^\dagger$ (page \pageref{p:trans}).
\end{proof}

\begin{lemma}\label{lm:enfin}
For each $S\in\ra_1$ there is $\omega\in\ca^\dagger$ such that 
$\inf_{\vkappa\in\ca^\dagger}\nu(S_\vkappa)=\nu(S_\omega)$. 
\end{lemma}

\begin{proof}
  Let $\{\varepsilon_i\}_{i\geq0}$ be a sequence of strictly positive
  numbers such that $\sum_i\varepsilon_i<\infty$ and let us set
  $\eta_m=\sum_{i\geq m}\varepsilon_i$. By Lemma \ref{lm:nuca}, there
  is a sequence of numbers $\theta_i$ such that
\[
\nu_{\theta_i}(S_\vkappa|\Omega)\leq\nu(S_\vkappa|\Omega)+\varepsilon_i  
\quad \forall\, \vkappa\in\ca^\dagger, \ \forall\, i, 
\text{ and } \forall\, \Omega\subset X \text{ open} .
\]
Then choose a sequence of points $\omega_n\in\ca^\dagger$ such that
$\lim_n\nu(S_{\omega_n})=\inf_{\vkappa\in\ca^\dagger}\nu(S_\vkappa)$. To
simplify notations we set $S_n=S_{\omega_n}$, so
$\lim_n\nu(S_n)=\inf_{\vkappa\in\ca^\dagger}\nu(S_\vkappa)$.  Now we
apply Lemma \ref{lm:imppp} to the operator $S_n$ which satisfies the
preceding inequality for $0\leq i\leq n$. If we set
$\zeta_m=\sum_{1\leq i\leq m}\theta_i+\theta_m$ we see that there is
a translate $T_n$ of $S_n$ such that
\[
\nu(T_n|B(\zeta_{m})) <\nu(S_n)+\eta_{m} \quad\forall\ 0\leq m\leq n.
\] 
Lemma \ref{lm:topology} implies that the sequence $\{T_{n}\}$ has a
subsequence $\{T_{n_k}\}$ convergent in $\rb_\loc$ to some
$S_\omega$ with $\omega\in\ca^\dagger$.  We have
\[
\nu(T_{n_k}|B(\zeta_{m})) <\nu(S_{n_k})+\eta_{m} 
\quad\forall\ 0\leq m\leq {n_k}
\] 
and $\lim_k T_{n_k}1_{B(\zeta_{m})}=S_\omega1_{B(\zeta_{m})}$ \emph{in
  norm} by the definition of the local norm topology. This implies
$\lim_k\nu(T_{n_k}|B(\zeta_{m}))=\nu(S_\omega|B(\zeta_{m}))$.  But
$\lim_k\nu(S_{n_k})=\inf_{\vkappa\in\ca^\dagger}\nu(S_\vkappa)$ by our
initial choice, hence we obtain
\[
\nu(S_\omega)\leq \nu(S_\omega|B(\zeta_{m}))\leq
\inf_{\vkappa\in\ca^\dagger}\nu(S_\vkappa) +\eta_m
\]
for any $m$. Making $m\to\infty$ we get
$\nu(S_\omega)\leq \inf_{\vkappa\in\ca^\dagger}\nu(S_\vkappa)$ which
finishes the proof.
\end{proof}

\section{Applications.}\label{s:appl}
\protect\setcounter{equation}{0}

This section is devoted to applications of Theorem \ref{th:GIA} in the
context of differential operators. In \S\ref{ss:unbounded} we develop
some tools which allow one to extend Theorem \ref{th:GIA} to unbounded
not necessarily self-adjoint operators. Then in \S\ref{ss:diff} we
consider singular elliptic differential operators affiliated to
$\re$. In \cite{GI4} one may find many other examples of algebras
$\ra$ and in each of these examples our Theorem \ref{th:GIA} can be
applied and gives a significant improvement of the results. For
example, if $X=\R^d$, the condition that $H$ be a normal operator in
Proposition 6.3 or Theorems 6.13 and 6.27 from \cite{GI4} is
eliminated.  Note that in these three examples we have been able to
show there, by ad hoc arguments, that the union is already a closed
set, but the non normal case was clearly not accessible. The
$C^*$-algebra involved in Theorem 6.27 in \cite{GI4} is the ``usual''
$N$-body algebra, i.e.\ the smallest $C^*$-algebra to which the
$N$-body Hamiltonians with $2$-body interactions tending to zero at
infinity are affiliated. But for the larger $C^*$-algebra generated by
$N$-body Hamiltonians with asymptotically homogeneous $2$-body
interactions introduced in \cite{GN} both questions (closedness of the
union and not normal case) remained open; they are solved in
\S\ref{ss:hom}.

\subsection{Unbounded operators}\label{ss:unbounded}

We extend here to non self-adjoint operators the notion of affiliation
to $C^*$-algebras developed in \cite{ABG,DG} and references therein.

For physical and technical reasons we are forced to consider non
densely defined self-adjoint operators.  In fact, even in the simplest
physically interesting case of operators of the form $H=p^2+v$ on
$L^2(\R)$ with $v$ a positive unbounded function some localizations at
infinity are not densely defined operators, cf.\
\S\ref{sss:unbounded}.

There are three natural ways of viewing self-adjoint operators
affiliated to a $C^*$-algebra $\rc$: as morphisms $C_0(\R)\to\rc$, as
resolvents, or as usual self-adjoint operators living in closed
subspaces of $\ch$ \cite[\S8.1.2]{ABG}.  If we think of $H$ as the
Hamiltonian of a quantum system, the morphism point of view is the
most natural one, but it has not an obvious extension to non
self-adjoint operators. So we shall treat the non self-adjoint case by
using the resolvent approach. Under a supplementary condition which
suffices in our applications we will also define an associated
operator which lives in a closed subspace of $\ch$.

\subsubsection{Closed operators}\label{sss:RH}

To justify later definitions, we recall some  facts concerning
a closed operator $H$ acting in $\ch$. Let $D(H)$ be its domain
equipped with the graph topology. Its resolvent set is
$\Omega(H)=\{z\in\C\mid H-z:D(H)\to\ch\text{ is bijective}\}$ and its
resolvent is the map $R:z\mapsto(H-z)^{-1}$ with domain $\Omega(H)$.
Let $R_{za}\doteq(H-z)R(a)=1-(z-a)R(a)$ for $a\in\Omega(H)$ and 
$z\in\C$; this is a bijective map $\ch\to\ch$ if $a,z\in\Omega(H)$.

The spectrum of $H$ is the set $\spec(H)=\C\setminus\Omega(H)$ and the
essential spectrum of $H$ is defined as the essential spectrum of the
continuous operator $H:D(H)\to\ch$, i.e.\ the set of numbers $\lambda$
such that $H-\lambda:D(H)\to\ch$ is not Fredholm. So
$\lambda\in\spe(H)$ means that either $\lambda$ is an eigenvalue of
infinite multiplicity of $H$ or the range of $H-\lambda$ is not of
finite codimension in $\ch$.  If $a\in\Omega(H)$ then we have the
following \emph{spectral mapping theorem}:
\begin{align}
\spec(H) &=\{\lambda\in\C \mid (\lambda-a)^{-1}\in\spec(R(a))\} ,
\label{eq:spmH} \\
\spe(H) &=\{\lambda\in\C \mid (\lambda-a)^{-1}\in\spe(R(a)) \}.
\label{eq:smH}
\end{align}
\emph{$H$ is affiliated to a $C^*$-algebra $\rc$} of operators on
$\ch$ if its resolvent set is not empty and $(H-z)^{-1}\in\rc$ for
some, hence all, $z\in\Omega(H)$ (cf.\ the text just above Theorem
\ref{th:GIU}).

\subsubsection{Resolvents}

We call \emph{resolvent}%
\footnote{ The usual terminology is ``pseudo-resolvent'' but in our
  context, and especially in the case of self-adjoint operators, it is
  unnatural to emphasize the distinction between densely and non
  densely defined operators.  }
any map $R: \Omega\to B(\ch)$ with $\Omega\subset\C$ not empty
satisfying $R(a)-R(b)=(a-b)R(a)R(b)$ for all $a,b\in \Omega$.
Equivalently:
\begin{equation}\label{eq:resid}
(1-(b-a)R(a))(1-(a-b)R(b))=1 \quad
\forall a,b\in \Omega .
\end{equation}
If we denote $R_{za}=1-(z-a)R(a)$ for $a\in\Omega$ and $z\in\C$ then
the preceding condition may be written
$R_{ab}R_{ba}=1\ \forall a,b\in\Omega$. Equivalently, $R_{ab}$ is
invertible and $R_{ab}^{-1}=R_{ba}$ if $a,b\in\Omega$. One may easily
check that for such $a,b$ we have
$R_{za}R_{ab}=R_{zb}\ \forall z\in\C$. Moreover, we clearly have
$R(z) =R(a)R_{az}=R(a) R_{za}^{-1}$ if $a,z\in\Omega$.

$R$ is called \emph{maximal} if there is no resolvent which extends
$R$ to a strictly larger domain. Fix some $a\in\Omega$. According to
\cite[Th.\ 5.8.6]{HP}, \emph{each resolvent $R$ has a unique maximal
  extension, the domain of this extension is the open set consisting
  of all $z\in\C$ such that $R_{za}$ is invertible, and the value of
  the extension at such a $z$ is $R(a) R_{za}^{-1}$} (the proof is
easy by the comments above).  We keep the notation $R$ for the maximal
extension of $R$.

Following \cite[\S5]{D}, we define the \emph{spectrum of the
  resolvent} $R$ as the complement of the domain of its maximal
extension, in other terms
\begin{equation}\label{eq:spres}
  \spec(R)=\{z\in\C\mid R_{za}:\ch\to\ch
  \text{ is not bijective}\}. 
\end{equation}
Then the \emph{essential spectrum} of $R$ is the subset of the
spectrum defined by
\begin{equation}\label{eq:speres}
  \spe(R)=\{z\in\C\mid R_{za}:\ch\to\ch
  \text{ is not Fredholm}\} .
\end{equation}
These definitions are independent of the choice of $a$: if, for
example, $R_{za}$ is Fredholm then $R_{zb}=R_{za}R_{ab}$ is Fredholm
too because $R_{ab}$ is invertible. Then by taking into account the
expression of $R_{za}$ we get for any $a\in\Omega$
\begin{align}
\spec(R) & =\{z\in\C\mid (z-a)^{-1}\in\spec(R(a))\},
\label{eq:spmtr} \\
\spe(R) & =\{z\mid(z-a)^{-1}\in\spe(R(a))\}.
\label{eq:smtr}
\end{align}
Clearly $R$ is a restriction of the resolvent of a closed operator $H$
as in \S\ref{sss:RH} if and only if $R(z):\ch\to\ch$ is injective for
some, hence for all, $z\in\Omega$.  Then from \eqref{eq:spmH} and
\eqref{eq:smH} we get $\spec(R)=\spec(H)$ and $\spe(R)=\spe(H)$.

Now let $\rc$ be a $C^*$-algebra of operators on $\ch$ and
$R:\Omega\to B(\ch)$ a resolvent. We say that $R$ is \emph{affiliated}
to $\rc$ if $R$ is $\rc$-valued, i.e.\
$R(z)\in\rc\ \forall z\in\Omega$. Note that for this it suffices that
$R(a)\in\rc$ for some $a\in\Omega$. Indeed, then $R_{za}\in\rc+\C$
hence its inverse $R_{az}$ also belongs to $\rc+\C$ so $R(z)\in\rc$.
Notice that if $\pi:\rc\to\rd$ is a morphism from $\rc$ to a
$C^*$-algebra $\rd$ then $\pi\circ R$ will be a resolvent affiliated
to $\rd$.

\begin{theorem}\label{th:GIU}
  Let $\ca$ be as in \emph{\S\ref{ss:alg}} and let $R:\Omega\to\ra$ be
  a resolvent.  Then for any $z\in\Omega$ and any $\vkappa\in\ca^\dag$
  the limit
  $\lim_{x\to\vkappa}\e^{\rmi xp}R(z)\e^{-\rmi xp}\doteq R_\vkappa(z)$
  exists locally in norm and defines a resolvent
  $R_\vkappa:\Omega\to\re$.  We have
\begin{equation}\label{eq:GIU}
\spe (R)= \ccup_{\vkappa\in\ca^\dagger} \spec(R_\vkappa).
\end{equation} 
\end{theorem}

\begin{proof}
  The existence of the limit and the fact that $R_\vkappa$ is a
  resolvent affiliated to $\re$ are consequences of Theorem
  \ref{th:GIold}. Then
  $\spe (R(a))= \ccup_{\vkappa\in\ca^\dagger} \spec(R_\vkappa(a))$ for
  any $a\in\Omega$ by Theorem \ref{th:GIA}. The relation
  \eqref{eq:GIU} follows from \eqref{eq:spmtr} and \eqref{eq:smtr}.
\end{proof}

\subsubsection{Regular resolvents}\label{sss:regular}

Our next purpose is to express Theorem \ref{th:GIU} in terms of an
operator $H$ such that $R(z)=(H-z)^{-1}$ in a generalized sense.  As
we already mentioned, even if we assume that $R$ is the resolvent of
an operator as in \S\ref{sss:RH}, the resolvents $R_\vkappa$ will in
general not be resolvents of operators in this sense and we need this
generalization to treat them.

Consider a resolvent $R:\Omega\to B(\ch)$.  Then the subspace
$R(z)\ch$ is independent of $z$ and we denote $\ch_R$ its closure.
The closed subspace $\cn_R=\ker R(z)$ is also independent of $z$.  And
$R$ is a restriction of the resolvent of a closed operator if and only
if $\cn_R=\{0\}$ and this operator is densely defined if and only if
$\ch_R=\ch$.  We say that $R$ is a \emph{regular resolvent} if
\begin{equation}\label{eq:kato}
  \exists z_n\in \Omega \text{ with } |z_n|\to \infty \text{ such
    that } \|z_nR(z_n)\|\leq\text{ const} .
\end{equation}
If $R$ is a regular resolvent affiliated to a $C^*$-algebra $\rc$ and
if $\pi:\rc\to\rd$ is a morphism, then $\pi\circ R$ is regular and
affiliated to $\rd$.  The following fact has been proved in \cite{K}.

\begin{lemma}\label{lm:kato}
  If $R$ is a regular resolvent then $\ch_R\cap\cn_R=\{0\}$ and
  $\ch=\ch_R+\cn_R$. 
\end{lemma}

In particular: \emph{if $R$ is regular and $\cn_R=\{0\}$ then there is
  a closed densely defined operator $H$ in $\ch$ such that
  $R(z)=(H-z)^{-1}$ for all $z\in \Omega$}. 

Since $R(z)\ch_R\subset\ch_R$ for any $z\in\Omega$, the restrictions
$R^\circ(z)=R(z)|\ch_R$ define a resolvent in $\ch_R$ and the range of
$R^\circ(z)$ is dense in $\ch_R$. If $R$ is regular then $R^\circ$
satisfies \eqref{eq:kato} hence is regular, so there is a closed
densely defined operator $H$ in the Hilbert space $\ch_R$ such that
$R^\circ(z)=(H-z)^{-1}$ for all $z\in\Omega$. Obviously the domain of
$H$ is $D(H)=R(z)\ch$ and we may, and we shall, think of $H$ as a
closed operator in $\ch$ such that $HD(H)$ is contained in the closure
$\ch_R$ of $D(H)$ in $\ch$. We say that $H$ is the \emph{operator
  associated to the regular resolvent} $R$.  Note that the resolvent
$R$ is not completely defined by its associated operator $H$, one must
also specify a closed subspace $\cn_R$ supplementary to
$\ch_R=\overline{D(H)}$.

\begin{example}{\rm The operator associated to the resolvent $R=0$
    with domain $\Omega=\C$ is denoted $H=\infty$. We have
    $D(H)=\{0\}$ and $\sigma(H)=\emptyset$.
}\end{example}

\begin{lemma}\label{lm:assoc}
Let $R$ be a regular resolvent and $H$ the operator associated to
it. Denote $\spec(H)$ and $\spe(H)$ the spectrum and the essential
spectrum of $H$ considered as operator in the Hilbert space
$\ch_R$. Then $\spec(R)=\spec(H)$ and $\spe(R)=\spe(H)$.
\end{lemma}

\begin{proof}
  We have already shown this in case $\cn_R=\{0\}$, so we may assume
  $\cn_R$ is not trivial. Relatively to the topological direct sum
  (but not orthogonal in general) decomposition $\ch=\ch_R\oplus\cn_R$
  we have $R(z)=R^\circ(z)\oplus0$. As mentioned after
  \eqref{eq:spmtr} and \eqref{eq:smtr} we have
  $\spec(R^\circ)=\spec(H)$ and $\spe(R^\circ)=\spe(H)$. Clearly
  $R_{za}=R^\circ_{za}\oplus1$ hence $R_{za}:\ch\to\ch$ is bijective
  if and only if $R^\circ_{za}:\ch_R\to\ch_R$ is bijective, so
  $\spec(R)=\spec(R^\circ)$. We also have
  $\ker(R_{za})=\ker(R^\circ_{za})\oplus\{0\}$ and
  $R_{za}\ch=R^\circ_{za}\ch_R\oplus\cn_R$ hence $R_{za}$ is Fredholm
  if and only if $R^\circ_{za}$ is Fredholm, so $\spe(R)=\spe(H)$.
\end{proof}

Let $R$ be a regular resolvent affiliated to $\ra$ and let $H$ be the
operator associated to it. If $R_\vkappa$ is as in Theorem
\ref{th:GIU} then $R_\vkappa$ is a regular resolvent affiliated to
$\re$ and we denote $H_\vkappa$ the operator associated to it. Then
the relation $\ulim_{x\to\vkappa}\e^{\rmi xp}H\e^{-\rmi xp}=H_\vkappa$
is an abbreviation for
``$\lim_{x\to\vkappa}\e^{\rmi xp}R(z)\e^{-\rmi xp}\doteq R_\vkappa(z)$
locally in norm for any $z\in\Omega,\vkappa\in\ca^\dag$''. Then
according to Theorem \ref{th:GIU} and Lemma \ref{lm:assoc} we have
\begin{equation}\label{eq:GIUH}
\spe (H)= \ccup_{\vkappa\in\ca^\dagger} \spec(H_\vkappa).
\end{equation} 

\subsubsection{Self-adjoint resolvents}

Finally, let us summarize these results in the self-adjoint case (see
\cite[\S8.1.2]{ABG} for detailed proofs).  \emph{A resolvent $R$ is
  called self-adjoint} if $\Omega$ is stable under conjugation and
$R(z)^*=R(\bar{z})$ $\forall z\in\Omega$. This implies
$\cn_R=\ch_R^\perp$.  Then there is a unique morphism
$\Phi:C_0(\R)\to B(\ch)$ such that $\Phi(r_z)=R(z)$ for all
$z\in\Omega$, where $r_z(\lambda)=(\lambda-z)^{-1}$ for
$\lambda\in\R$. The map $R\mapsto\Phi$ is a bijective correspondance
between self-adjoint resolvents on $\ch$ and morphisms
$C_0(\R)\to B(\ch)$.  A self-adjoint resolvent is clearly regular and
the operator associated to it is a densely defined self-adjoint
operator in the Hilbert space $\ch_R$.

Let us call \emph{observable on $\ch$} any densely defined
self-adjoint operator acting in a closed subspace of $\ch$.  The map
$R\mapsto H$ is a bijective correspondance between self-adjoint
resolvents on $\ch$ and observables on $\ch$.  Thus we may identify
observables $H$, self-adjoint resolvents $R$, and morphisms $\Phi$. An
observable is affiliated to $\rc$ if $R$ is affiliated to $\rc$, or if
$\Phi:C_0(\R)\to\rc$.

If $H$ is an observable affiliated to $\ra$ then
$\ulim_{x\to\vkappa}\e^{\rmi xp}\Phi(u)\e^{-\rmi xp}= \Phi_\vkappa(u)$
exists for any $u\in C_0(\R)$, by the Stone-Weierstrass theorem.  Now
if we apply the preceding results with $\ra=\re$ we get:

\begin{theorem}\label{th:GIH}
  Let $H$ be an observable on $L^2(X)$ such that for some number $z$
  in its resolvent set its resolvent $R(z)=(H-z)^{-1}$ satisfies
\begin{equation}\label{eq:GIH}
  \lim_{a\to0}\|(\e^{\rmi ap}-1)R(z)\|=0
  \text{ and }
  \lim_{a\to0}\|[\e^{\rmi aq}, R(z)]\|=0.
\end{equation}
Then for any $\vkappa\in X^\dagger$ the limit
$\ulim_{x\to\vkappa} \e^{\rmi xp}H\e^{-\rmi xp}\doteq H_\vkappa$
exists and
\begin{equation}\label{eq:GIH2}
\spe (H)= \ccup_{\vkappa\in X^\dagger} \spec(H_\vkappa).
\end{equation} 
\end{theorem}

\subsection{Differential operators}\label{ss:diff}

We now present some applications of the abstract results from
\S\ref{ss:mresults} and \S\ref{ss:unbounded} to differential operators
emphasising those of interest in quantum mechanics.  In fact all we
have to do is to give examples of operators affiliated to the algebra
$\ra$ and which are of some independent interest. In the maximal case
$\ra=\re$ this is easy because the conditions of affiliation to
$\re$ are very explicit and easy to check. For other $\ra$ one may
use the perturbative affiliation criteria developed in \cite{DG}, see
for example Theorems 2.5 and 2.8 there (these results can be extended
to non self-adjoint operators in a rather obvious way, but we shall
not discuss this topic here).  For some $\ra$ there are criteria of
affiliation of the same nature as in the case of $\re$, e.g.\
\cite[Th.\ 5.2]{GN}.

\subsubsection{Non self-adjoint operators}

We consider first a class of non self-adjoint differential operators
affiliated to $\re$. For simplicity of the presentation we consider
only relatively bounded perturbations of the Laplacian, the extension
to more general operators requires more formalism \cite[\S6.6]{GI4}.
More singular operators will be treated in the self-adjoint case.

We recall a general fact.  Let $\ch$ be a Hilbert space and $\cg$ a
dense subspace equipped with a Hilbert structure such that the
embedding $\cg\hookrightarrow\ch$ is continuous. We identify the
adjoint space $\ch^*=\ch$ via the Riesz isomorphism and then embed
$\cg\hookrightarrow\ch\hookrightarrow\cg^*$ as usual.  The operator
$\what{T}$ in $\ch$ associated to, or induced by, $T\in B(\cg,\cg^*)$
is the restriction of $T$ to $T^{-1}(\ch)=\{u\in\cg\mid Tu\in\ch\}$
considered as operator in $\ch$.  If $z\in\C$ then
$\what{T-z}=\what{T}-z$. One may easily check that if there is $z$
such that $T-z:\cg\to\cg^*$ is bijective then $\what{T}$ is a densely
defined closed operator in $\ch$ such that $z$ belongs to its
resolvent set and $\what{T}^*=\what{T^*}$.

From now on, unless otherwise explicitly stated, we take $\ch=L^2(X)$
with norm $\|\cdot\|$ and for real $s$ we denote $\ch^s\equiv\ch^s(X)$
the usual Sobolev space with norm $\|\cdot\|_s$.

Consider an operator $V\in B(\ch^1,\ch^{-1})$ such that
$\Re V\geq-\mu\Delta-\nu$ as forms on $\ch^1$ for some numbers
$\mu,\nu$ with $\mu<1$; here $\Re V=(V+V^*)/2$.  Then
$\Delta+V:\ch^1\to\ch^{-1}$ is a continuous operator whose adjoint is
$\Delta+V^*:\ch^1\to\ch^{-1}$. If $u\in\ch^1$ and $\lambda>\nu$ then
\begin{equation}\label{eq:DV}
\Re\braket{u}{(\Delta+V+\lambda)u} \geq
\braket{u}{[(1-\mu)\Delta+(\lambda-\nu)]u} \geq c\|u\|_1^2
\end{equation}
for a constant $c>0$. This implies
$c\|u\|_1\leq\|(\Delta+V+\lambda)u\|_{-1}$ and the same estimate with
$V$ replaced by $V^*$. Hence $\Delta+V+\lambda:\ch^1\to\ch^{-1}$ is
bijective if $\lambda>\nu$ so the operator $H$ associated to
$\Delta+V$ in $\ch$ is closed, densely defined, the half-line
$]-\infty,-\nu[$ is included in the resolvent set of $H$, and $H^*$ is
the operator induced by $\Delta+V^*$ in $\ch$.  Moreover from
\eqref{eq:DV} we get for $u\in D(H)$
\[
(\lambda-\nu)\|u\|^2\leq \Re\braket{u}{(\Delta+V+\lambda)u}\leq
\|u\| \|(\Delta+V+\lambda)u\|              
\]
hence the resolvent of $H$ satisfies
$(\lambda-\nu)\|R(-\lambda)\|\leq1$ for $\lambda>\nu$, so is regular.

\begin{theorem}\label{th:nsad}
  Let $V\in B(\ch^1,\ch^{-1})$ such that
  $\lim_{a\to0}\|[\e^{\rmi aq},V]\|_{\ch^1\to\ch^{-1}}=0$ and
  $\Re V\geq-\mu\Delta-\nu$ with $\mu<1$.  Then the operator $H$ in
  $\ch$ associated to $\Delta+V$ is affiliated to $\re$ hence
  $\spe(H)=\cup_{\vkappa\in X^\dagger}\spec(H_\vkappa)$.
\end{theorem}

\begin{proof}
  We are here in the setting of \S\ref{sss:regular} and the formula
  for the essential spectrum has to be interpreted as in
  \eqref{eq:GIUH}. We just have to prove that $H$ is affiliated to
  $\re$ and for this it suffices to show that for some
  $\lambda>\nu$ 
\begin{equation}\label{eq:GIHV}
  \lim_{a\to0}\|(\e^{\rmi ap}-1)R(\lambda)\|=0 \text{ and } 
  \lim_{a\to0}\|[R(\lambda),\e^{\rmi aq}]\|=0. 
\end{equation}
According to \eqref{eq:ellalg} we should also prove that the first
condition above is satisfied with $R(\lambda)$ replaced by
$R(\lambda)^*$, but this is obvious here because $R(\lambda)^*$ is the
resolvent of the operator $H^*$ which involves $V^*$ and $V^*$
satisfies the same conditions as $V$. The first condition in
\eqref{eq:GIHV} is satisfied because the range of $R(\lambda)$ is
included in $\ch^1$. For the second one we note that $\ch^1$ is stable
under $\e^{\rmi aq}$ so if we denote $\|\cdot\|_{B}$ the norm in
$B(\ch^{-1},\ch^1)$ then
\[
\|[\e^{\rmi aq},R(\lambda)]\|_B=
\|R(\lambda) [V,\e^{\rmi aq}] R(\lambda)\|_B\leq
\|R(\lambda)\|_B^2 \cdot \|[V,\e^{\rmi aq}]\|_{\ch^1\to\ch^{-1}} .
\] 
Thus $\|[\e^{\rmi aq},R(\lambda)]\|_B\to0$ which is more than needed.
\end{proof}

Under stronger conditions on $V$ we have a more explicit description of the
operators $H_\vkappa$. We will say that a function $V:X\to\C$ is
\emph{$\Delta$-small} if $V$ is locally integrable and for each
$\mu>0$ there is $\nu$ such that $|V|\leq\mu\Delta+\nu$. Also, a
symmetric operator $S:\ch^1\to\ch^{-1}$ is $\Delta$-small if for each
$\mu>0$ there is $\nu$ such that $\pm S\leq\mu\Delta+\nu$.

\begin{theorem}\label{th:nsad2}
  Let $V:X\to\C$ be $\Delta$-small.  Then
  $\lim_{x\to\vkappa}V(x+q)\doteq V_\vkappa$ exists in the strong
  topology in $B(\ch^1,\ch^{-1})$ for each $\vkappa\in X^\dagger$, the
  operators $\Re V_\vkappa:\ch^1\to\ch^{-1}$ are $\Delta$-small, and
  $H_\vkappa$ is the operator in $\ch$ associated to
  $\Delta+V_\vkappa:\ch^1\to\ch^{-1}$.
\end{theorem}

\begin{proof}
  If $W$ is the real or imaginary part of $V$ then $W$ is
  $\Delta$-small hence by Proposition 6.33 from \cite{GI4} the
  translates $W(x+q)$ of the operator $W(q)\in B(\ch^1,\ch^{-1})$
  converge strongly to some $W_\vkappa \in B(\ch^1,\ch^{-1})$ hence
  the same holds for the translates $V(x+q)$.  Obviously an estimate
  like $\pm\Re V(x+q)\leq\mu\Delta+\nu$ remains valid in the limit
  $x\to\vkappa$, hence $\Re V_\vkappa:\ch^1\to\ch^{-1}$ is
  $\Delta$-small. If $K_\vkappa$ is the operator in $\ch$ associated
  to $\Delta+V_\vkappa$, it remains to prove that
  $K_\vkappa=H_\vkappa$. Fix $\mu<1$ and let $\nu$ such that
  $|V|\leq\mu\Delta+\nu$. Clearly, if $\lambda>\nu$ then $-\lambda$ is
  in the resolvent set of $K_\vkappa$ and of
  $H_x=\e^{\rmi xp}H\e^{-\rmi xp}=\Delta+V_x$ for any $x\in X$, where
  $V_x=V(x+q)$. Then
\[
(\Delta+V_x+\lambda)^{-1}-(\Delta+V_\vkappa+\lambda)^{-1}=
(\Delta+V_x+\lambda)^{-1}(V_\vkappa-V_x)(\Delta+V_\vkappa+\lambda)^{-1}
\]
holds in $B(\ch^{-1},\ch^1)$, hence for $u\in\ch^{-1}$ we have
\begin{align*}
& \|\big((\Delta+V_x+\lambda)^{-1}
-(\Delta+V_\vkappa+\lambda)^{-1}\big)u\|_1 \\
& \leq \|(\Delta+V+\lambda)^{-1}\|_B
\|(V_\vkappa-V_x)(\Delta+V_\vkappa+\lambda)^{-1}u\|_{-1}
\end{align*}
where $\|\cdot\|_{B}$ is the norm in $B(\ch^{-1},\ch^1)$. The last
factor above converges to zero as $x\to\vkappa$ hence
$(\Delta+V_x+\lambda)^{-1}\to(\Delta+V_\vkappa+\lambda)^{-1}$ strongly
in $B(\ch^{-1},\ch^1)$, which clearly implies
$(H_x+\lambda)^{-1}\to (K_\vkappa+\lambda)^{-1}$ strongly in
$B(\ch)$. But
$\lim_{x\to\vkappa}\e^{\rmi xp}R(-\lambda)\e^{-\rmi xp}=
R_\vkappa(-\lambda)$ locally in norm, hence
$R_\vkappa(-\lambda)=(K_\vkappa+\lambda)^{-1}$. Thus
$H_\vkappa=K_\vkappa$. 
\end{proof}

\subsubsection{Affiliation criteria}

We discuss here general conditions which ensure the affiliation to
$\re$ of self-adjoint operators \cite[\S4.2]{GI4} and in the next
paragraphs give some concrete examples.  In each of these cases one
may use Theorem \ref{th:GIH}.

Recall that a self-adjoint operator $H$ on $\ch\equiv L^2(X)$ is
affiliated to $\re$ if and only if for some $z\nin\spec(H)$ the
operator $R(z)=(H-z)^{-1}$ satisfies
 \begin{equation}\label{eq:GIR}
   \lim_{a\to0}\|(\e^{\rmi ap}-1)R(z)\|=0
   \text{ and } \lim_{a\to0}\|\e^{-\rmi aq}R(z)\e^{\rmi aq}-R(z)\|=0. 
 \end{equation}
 The first condition above is equivalent to the existence of a
 continuous function $\phi:X\to\R$ with
 $\lim_{x\to\infty}\phi(x)=+\infty$ such that $D(H)\subset D(\phi(p))$.
 For example, it suffices to have $D(H)\subset\ch^s$ for some $s>0$.
 The second condition in \eqref{eq:GIR} is a sort of regularity
 condition on the dependence on $p$ of $H$. One could check it by
 justifying the writing
 \begin{align*}
 \e^{-\rmi aq}R(z)\e^{\rmi aq}-R(z) 
 &=(\e^{-\rmi aq}H\e^{\rmi aq}-z)^{-1}-(H-z)^{-1} \\
 &=-(\e^{-\rmi aq}H\e^{\rmi aq}-z)^{-1}(\e^{-\rmi aq}H\e^{\rmi aq}-H)(H-z)^{-1}
 \end{align*}
 and imposing a condition on $\e^{-\rmi aq}H\e^{\rmi aq}-H$. In the
 simplest case $H=h(p)+v(q)$ with some real function $h$ we have
 $\e^{-\rmi aq}H\e^{\rmi aq}-H= h(p+a)-h(p)$ so one is forced to
 impose a smoothness condition on the function $h$.

From now on in this subsection we use the same notation for a function
on $X$ and for the operator of multiplication by this function in
function spaces on $X$, e.g.\ $V\equiv V(q)$. 

Let $h:X\to\R$ with $\lim_{k\to\infty}h(k)=+\infty$.  Let $m>0$ an
integer and assume that $h$ is of class $C^m$, its derivatives of
order $m$ are bounded, and $|h^{(\alpha)}(k)|\leq C(1+|h(k)|)$ if
$|\alpha|\leq m$. For example, $h$ could be a hypoelliptic polynomial,
or $h(k)=\sqrt{k^2+1}$, etc.  

Then $h(p)$ is a self-adjoint operator on $\ch$. Denote
$\cg=D(|h(p)|^{1/2})$ its form domain equipped with the graph topology
and $\cg^*$ the space adjoint to $\cg$. As usual
$\cg\subset\ch\subset\cg^*$.

Clearly $\e^{\rmi aq}\cg=\cg$ and $\e^{\rmi aq}$ extends to a bounded
operator on $\cg^*$ for which we keep the notation $\e^{\rmi aq}$.
Thus $(\e^{\rmi aq})_{a\in X}$ is a $C_0$-group in each of the spaces
$\cg,\ch,\cg^*$ and the commutator $[\e^{\rmi aq},S]$ is a well
defined element of $B(\cg,\cg^*)$ if $S\in B(\cg,\cg^*)$.

Let $W:\cg\to\cg^*$ be a symmetric operator satisfying
$W\geq-\mu h(p)-\nu$ for some real numbers $\mu,\nu$ with
$\mu<1$. Then the sum $h(p)+W$ is a well defined operator
$\cg\to\cg^*$ and the operator associated to it in $\ch$ is a
self-adjoint bounded from below operator that we denote $H_0$.  Assume
$\lim_{a\to0}\|[\e^{\rmi aq},W]\|_{\cg\to\cg^*}=0$. \emph{Then $H_0$
  is affiliated to $\re$.}

Now let $V:X\to\R$ be a locally integrable function and let $V_-$ be
its negative part. Assume that there are numbers $\mu,\nu$ with
$\mu<1$ such that $V_-\leq\mu H_0+\nu$ in form sense.

\emph{Then the self-adjoint operator $H$ associated to the form sum
  $H_0+V$ is affiliated to $\re$.}

\subsubsection{Uniformly elliptic operators}

We emphasize that the conditions on the perturbation $W$ considered
above are such that $W$ may contain terms of the same order as
$h(p)$. For example, the next fact is a consequence of the preceding
statement. Recall the notations $p_j=-\rmi\partial_j$ and
$p^\alpha=p_1^{\alpha_1}\dots p^{\alpha_d}$ for $\alpha\in\N^d$.

\emph{Let $a_{\alpha\beta}\in L^\infty(X)$ such that
  $L=\sum_{|\alpha|,|\beta|\leq m} p^\alpha a_{\alpha\beta}p^\beta$ is
  uniformly elliptic on $\ch^m$, i.e.\
  $\braket{u}{Lu}\geq \mu\|u\|^2_{\ch^m} -\nu\|u\|^2$ with
  $\mu,\nu>0$. If $V\in L^1_\loc(X)$ is positive then the self-adjoint
  operator $H$ associated to the form sum $L+V$ is affiliated to
  $\re$.}

\subsubsection{Schr\"odinger operators}

We consider now Schr\"odinger operators with singular potentials.
Note that $\Delta=p^2$ is the positive Laplacian.

\emph{Let $W$ be a continuous symmetric sesquilinear form on $\ch^1$
  such that $W\geq-\mu\Delta-\nu$ with $\mu<1$ and
  $\lim_{a\to0}\|[\e^{\rmi aq},W]\|_{\ch^1\to\ch^{-1}}=0$.  Denote
  $H_0$ the self-adjoint operator associated to the form $\Delta+W$ on
  $\ch^1$.  Let $V:X\to\R$ locally integrable such that its negative
  part is form bounded with respect to $H_0$ with relative bound
  $<1$. Then the self-adjoint operator $H$ associated to the form sum
  $H_0+V$ is affiliated to $\re$.}

These conditions are satisfied if $W=0$ and $V$ is of Kato class,
hence we see that closures are not needed in the Theorems 3.12 and 4.5
in the paper \cite{LaS} of Y.~Last and B.~Simon.

\subsubsection{Magnetic fields}

The next example involves magnetic fields and a perturbation
of the Euclidean metric. Let $L$ be the form sum 
\[
L={\textstyle\sum_{ij}}(p_i-a_i)g_{ij}(p_j-a_j)+v
\]
where $g_{ij},a_i,v$ are operators of multiplication by real functions
such that: \\[1mm]
1) $g_{ij}\in L^\infty(X)$ and the matrix $G(x)=(g_{ij}(x))$ is
uniformly positive definite, i.e.\ there is a number $\varepsilon>0$ such that
$G(x)\geq\varepsilon$ for all $x$; \\[1mm]
2) $\forall\mu>0$ $\exists\nu>0$ such that
$\sum_i\|a_iu\|^2+\braket{u}{w u} \leq \mu\|u\|^2_{\ch^1}+\nu\|u\|^2$
$\forall u\in \ch^1$, where $w$ is the negative part of $v$.\\[1mm]
\emph{Then the self-adjoint operator $H$ associated to the form sum
  $L$ is affiliated to $\re$.}

\subsubsection{Dirac operators}

The theory extends easily to operators in
$\ch=L^2(X,E)=L^2(X)\otimes E$ where $E$ is a finite dimensional
Hilbert space and covers singular Dirac operators \cite[Prop.\ 1.11
and Cor.\ 4.8]{GI4}. Now by $\re$ we understand $\re\otimes B(E)$.

Let $D$ be the Dirac operator acting in $\ch$. So $D$ is a first order
symmetric differential operator with constant coefficients realized as
a self-adjoint operator in $\ch$ with domain the Sobolev space
$\ch^{1}$. Then let $V$ be a continuous symmetric form on $\ch^{1/2}$
such that $\pm V \leq \mu|D|+\nu$ with $\mu<1$ and
$\lim_{a\to0}\|[\e^{\rmi aq},V]\|_{\ch^{1/2}\to\ch^{-1/2}}=0$.
\emph{Then the operator in $\ch$ associated to the sum
  $D+V:\ch^{1/2}\to\ch^{-1/2}$ is self-adjoint and affiliated to
  $\re$.}

\subsubsection{Stark Hamiltonian}\label{sss:stark}

There are physically interesting differential operators which are not
affiliated to $\re$. Indeed, \emph{the Stark Hamiltonian $H=p^2+q$ is
  a self-adjoint operator on $L^2(\R)$ which does not satisfy any of
  the conditions \eqref{eq:GIR}}. And we have
$\slim_{|a|\to\infty}\e^{\rmi ap}H\e^{-\rmi ap}=\infty$ and
$\slim_{|a|\to\infty}\e^{\rmi aq}H\e^{-\rmi aq}=\infty$, while
$\spe(H)=\R$.

\begin{proof}
  Note first that if $\theta\in L^\infty(\R)$ is a nonzero function
  then $\|(\e^{\rmi ap}-1)\theta(q)\|$ does not tend to zero as
  $a\to0$. Indeed, by a remark made just after \eqref{eq:GIR}, in the
  contrary case there is $\psi\in C_0(\R)$ such that
  $\theta(q)=\psi(p)T$ for some bounded operator $T$. Then for any
  $\eta\in C_0(\R)$ the operator $\eta(q)\theta(q)=\eta(q)\psi(p)T$ is
  compact, which is obviously wrong. 

  Let $R=(H+\rmi)^{-1}$ and $s=p^3/3$. Since
  $H=\e^{\rmi s}q \e^{-\rmi s}$ we have
  $R=\e^{\rmi s}(q+\rmi)^{-1} \e^{-\rmi s}$ hence
  $\|(\e^{\rmi ap}-1)R\|=\|(\e^{\rmi ap}-1)(q+\rmi)^{-1}\|$ which does
  not tend to zero as $a\to0$ by the preceding remark. Then a short
  computation gives
\[
\e^{\rmi aq}R\e^{-\rmi aq}=
\e^{\rmi s}(q-2ap+a^2+\rmi)^{-1}\e^{-\rmi s} 
\]
hence
\begin{align*}
\|\e^{\rmi aq}R\e^{-\rmi aq}-R\| &=
\|(q-2ap+a^2+\rmi)^{-1}-(q+\rmi)^{-1}\| \\
&=\|(q-2ap+\rmi)^{-1}-(q+\rmi)^{-1}\| +O(a^2).
\end{align*}
To show that this does not tend to zero as $a\to0$ it suffices to use
the estimate
\[
\|\varphi(q)+\psi(p+\alpha q)\|\geq \max\big(\sup|\varphi|,\sup|\psi|\big) 
\]
which is valid and easy to prove for any $\varphi,\psi\in
C_0(\R)$. Finally, note for example that
$\slim_{|a|\to\infty}\e^{\rmi aq}H\e^{-\rmi aq}=\infty$ means 
$\slim_{|a|\to\infty}\e^{\rmi aq}R\e^{-\rmi aq}=0$ which is an
exercise. 
\end{proof}

\subsubsection{Unbounded potentials}\label{sss:unbounded}

We give two one dimensional examples of Hamiltonians $H=p^2+v$
affiliated to $\re(\R)$ with potentials $v$ not relatively bounded
with respect to $p^2$.

The first one is such that one of its localizations at infinity is not
densely defined.  For $n>0$ integer let $v(x)=n$ if $n^2-n<x<n^2$ and
$v(x)=0$ if $n^2<x<n^2+n$. Let $v$ be an arbitrary bounded positive
function on $x<0$. Denote $R=(H+1)^{-1}$ and let $\Delta_+$ be the
Dirichlet Laplacian on $(0,\infty)$.  Then
$\ulim_{n\to\infty}\tau_{n^2}(R)=R_+$ where $R_+=(\Delta_++1)^{-1}$ on
$L^2(0,\infty)$ and $R_+=0$ on $L^2(-\infty,0)$. Thus one of the
localizations at infinity of $H$ is the Dirichlet Laplacian on
$(0,\infty)$ considered as a non densely defined self-adjoint operator
on $L^2(\R)$ which ``lives'' in the subspace $L^2(0,\infty)$.

We now give a second one dimensional example where the essential
spectrum has an interesting structure. Note that if $X=\R$ and
$\vkappa$ is an ultrafilter then either $]0,+\infty[\,\in\vkappa$ or
$]-\infty,0[\,\in\vkappa$ hence there are two contributions to the
essential spectrum of $H$, one coming from the behaviour at $+\infty$,
that we denote $\mathrm{Sp^+_{ess}}(H)$, and one coming from the
behaviour at $-\infty$, that we denote $\mathrm{Sp^-_{ess}}(H)$.

Let $H=h(p)+v(q)$, where $h:\R\to\R$ is of class $C^1$, polynomially
bounded, tends to $+\infty$ if $p\rarrow\infty$, and
$|h'(p)|\leq C(1+|h(p)|)$. The function $v$ is real, continuous, and
bounded from below.  Assume that for large positive $x$ we have
$v(x)=x^a\omega(x^\theta)$ with $a\geq0$, $0<\theta<1$ and $\omega$ a
positive continuous periodic function with period 1. Moreover, assume
that $\omega$ vanishes only at the points of $\Z$ and that there are
real numbers $\lambda,\mu>0$ such that $\omega(t)\sim \lambda|t|^\mu$
when $t\rarrow0$.  Then there are three possibilities:

{\rm (1)} If $a<\mu(1-\theta$) then
$\mathrm{Sp^+_{ess}}(H)= [\inf h,+\infty)$.

{\rm (2)} If $a=\mu(1-\theta$) then
$\mathrm{Sp^+_{ess}}(H)=\spec{}(h(p) + \lambda|\theta q|^\mu)$, a
discrete not empty set.

{\rm (3)} If $a>\mu(1-\theta$) then
$\mathrm{Sp^+_{ess}}(H)= \emptyset$.

This is a consequence of Corollary \ref{co:classic} and the proof
consists in a rather straightforward computation of the localizations
at infinity of $H$. In the case (1) these are all the operators
$h(p)+c$ with $c\geq0$ real while in case (2) they are of the form
$h(P) + \lambda|\theta q+c|^\mu$ with $c\in\R$. In the third situation
all the localizations at infinity are $\infty$ (see Proposition
\ref{pr:discrete}).

\subsubsection{Discrete spectrum}

A self-adjoint operator $H$ has purely discrete spectrum if and only
if $\spe(H)=\emptyset$.  If $H$ is affiliated to $\re$, i.e. if the
conditions of \eqref{eq:sadj} are satisfied, then due to
\eqref{eq:rb-aa} this means $H_\vkappa=\infty$ for any ultrafilter
$\vkappa$. Thus Corollary \ref{co:classic} implies: 

\begin{proposition}\label{pr:discrete}
  If $H$ is a self-adjoint operator affiliated to $\re$ then $H$ has
  purely discrete spectrum if and only if
  $\wlim_{a\to\infty}\e^{\rmi ap}R(z) \e^{-\rmi ap}=0$ for some $z$ in
  the resolvent set of $H$.
\end{proposition}

The next result is a consequence of this proposition, we refer to
\cite{G0} for the proof. See \cite{bS} for results of the same nature
obtained by other techniques. Let $B_a$ be the ball $|x-a|<1$.

\begin{corollary}\label{co:discrete}
  Let $H_0$ be a bounded from below self-adjoint operator with form
  domain equal to the Sobolev space $\ch^m$ for some real $m > 0$ and
  satisfying $\lim_{a\to0}\e^{\rmi aq}H_0\e^{-\rmi aq}= H_0$ in norm
  in $B(\ch^m,\ch^{-m})$. Let $V$ be a positive locally integrable
  function such that
  $\lim_{a\to\infty}|\{x\in B_a \mid V (x) < \lambda\}| = 0$ for each
  $\lambda> 0$. Then the self-adjoint operator $H$ associated to the
  form sum $H_0 + V$ has purely discrete spectrum.
\end{corollary}

\subsection{N-body Hamiltonians}\label{ss:hom}

The $C^*$-algebra involved in the preceding subsection was the maximal
one $\ca=C_\rmb^\rmu(X)$. We will consider now an algebra $\ca=\cR(X)$
introduced in \cite{GN} (with a different notation) which has an
interesting structure and character space so its study is worthwhile
independently of the quantum mechanical applications mentioned in
\cite{GN}. Theorem \ref{th:GIA} will allow us to solve a problem left
open in \cite{GN}.

The terminology ``N-body'' is misleading in a certain measure, for
example N has no significance, it should only suggest that the
operators affiliated to the algebra $\cR(X)\rtimes X$ include the
Hamiltonians describing quantum $N$-particle systems. But in fact the
usual such Hamiltonians are affiliated to a smaller algebra studied in
\cite[\S4]{DG} (and previous articles quoted there) with \cite[Th.\
6.27]{GI4} as analog of Theorem \ref{th:GIH}; note that the $N$-body
Hamiltonians with hard-core interactions studied in \cite{BGS} are
also affiliated to this algebra.  But the $N$-body type Hamiltonians
affiliated to the algebra $\cR(X)\rtimes X$ are much more general
since the allowed interactions are only required to be asymptotically
homogeneous.

Let $E$ be a finite dimensional real vector space. The half-line
determined by a vector $a\neq0$ in $E$ is the set
$\hat{a}=\{ra \mid r>0\}$, the sphere at infinity $\mbs_E$ of $E$ is
the set of all half-lines in $E$, and as a set the \emph{spherical
  compactification} of $E$ is the disjoint union
$\bar{E}=E\cup\mbs_E$.  This space has a natural compact space
topology \cite[\S3]{GN} such that $E\hookrightarrow\bar{E}$ as an open
dense subset hence $C(\bar{E})$ can be identified with an algebra of
continuous functions on $E$. More precisely, by restricting functions
on $\bar{E}$ to $E$ one may realize $C(\bar{E})$ as a
\emph{translation invariant} $C^*$-subalgebra of $C_\rmb^\rmu(E)$
which contains $C_\infty(E)$.

The simplest example of this type is $\ca=C(\ol{\R})$ the algebra of
continuous functions on $\R$ that have limits at $\pm\infty$. Then
$\ca^\dagger=\{-\infty,+\infty\}$ consists of just two points, for any
$A\in\ra$ the limits
$A_{\pm}=\slim_{x\to\pm\infty}\e^{\rmi xp}A \e^{-\rmi xp}$
exist, and $\spe(A)=\spec(A_-)\cup\spec(A_+)$. This is easy to prove
directly, without any formalism.

The higher dimensional analog $\ca=C(\oX)$ is less trivial.  Now
$\ca^\dagger=\mbs_X$ and for any $A\in\ra$ and any
$a\in\alpha\in\mbs_X$ the limit
$A_{\alpha}=\slim_{r\to\infty}\e^{\rmi rap}A \e^{-\rmi rap}$ exists
(and is clearly independent of the choice of $a$).  It has been shown
in \cite[Cor.\ 4.5]{GN} that if $A$ is a normal operator then
$\spe(A)=\cup_{\alpha\in\mbs_X}\spec(A_\alpha)$, in particular the
union is a closed set, but the techniques used in \cite{GN} are not
applicable to not normal operators.  We shall give now a complete
treatment of an N-body type generalization of this example.

Let $X$ be a real finite dimensional vector space (it is not
convenient to identify it with $\R^d$). If $Y\subset X$ is a linear
subspace and if we denote $\pi_Y:X\to X/Y$ the canonical surjection,
then the map $\phi\mapsto\phi\circ \pi_Y$ gives a natural embedding
$C_\rmb^\rmu(X/Y)\subset C_\rmb^\rmu(X)$ so we may, and we shall,
think of $C_\rmb^\rmu(X/Y)$ as a $C^*$-subalgebra of
$C_\rmb^\rmu(X)$. On the other hand, $C(\oXY)$ is a $C^*$-subalgebra
of $C_\rmb^\rmu(X/Y)$, so we obtain a realization of $C(\oXY)$ as a
$C^*$-subalgebra of $C_\rmb^\rmu(X)$. Finally, we define the algebra
$\ca$ of interest in this example:
\begin{equation}\label{eq:crx}
\cR(X)\doteq
\begin{cases}
\text{ $C^*$-algebra generated by the subalgebras } C(\oXY) \\
\text{ when } Y \text{ runs over the set of all subspaces of } X. 
\end{cases}
\end{equation}
It is easy to see that
$C_\infty(X)\subset\cR(X)\subset C_\rmb^\rmu(X)$ and that $\cR(X)$ is
stable under translations, so we may take $\ca=\cR(X)$ in Theorem
\ref{th:GIA}. The corresponding $\ra$ is the crossed product
$\rr(X)\doteq\cR(X)\rtimes X$ identified with the closed linear
subspace of\,%
\footnote{ We refer to \cite{GN} for details. Note however that the
  $C^*$-algebras $\rb(X)$, $\rk(X)$ do not depend on the choice of a
  Lebesgue measure on $X$. We denote $X^*$ the vector space dual to
  $X$ }
$\rb(X)$ generated by the operators $\varphi(q)\psi(p)$ with
$\varphi\in\cR(X)$ and $\psi\in C_0(X^*)$.

We shall keep the notation $\e^{\rmi ap}$ for the translation operator
by $a\in X$ although we do not give a meaning to the symbol $p$. This
operator acts as before: $(\e^{\rmi ap}u)(x)=u(a+x)$. To be precise,
$r\to\infty$ means $r\to+\infty$.

\begin{theorem}\label{th:GN}
  If $A\in\rr(X)$ and $\alpha\in\mbs_X$ then
  $\ulim_{r\to\infty}\e^{\rmi rap}A \e^{-\rmi rap}\doteq A_\alpha$
  exists for each $a\in\alpha$ and is independent of $a$.  We have
  $\spe (A)= \ccup_{\alpha\in\mbs_X} \spec(A_\alpha)$.
\end{theorem}

\begin{corollary}\label{co:GN}
  If $H$ is a self-adjoint operator affiliated to $\rr(X)$ then for
  any $\alpha\in\mbs_X$ the limit
  $H_\alpha\doteq\ulim_{r\to\infty}\e^{\rmi rap}H \e^{-\rmi rap}$
  exists for each $a\in\alpha$ and is independent of $a$.  We have
  $\spe (H)= \ccup_{\alpha\in\mbs_X} \spec(H_\alpha)$.
\end{corollary}

The first part of the theorem is a consequence of \cite[Th.\ 6.18]{GN}
and the corollary is an improvement of \cite[Th.\ 6.21]{GN}, where the
formula for the spectrum is shown with the union replaced by the
closure of the union. This question is discussed in \cite{GN} after
Theorem 6.21 and Lemma 6.22 there gives a hint about the difficulty of
the problem. On the other hand, the question of non self-adjoint $A$
is not at all discussed in \cite{GN} because the relation (6.20) from
\cite{GN} gives nothing for such operators.

Let us mention a recent result from \cite{MNP} which is related to
Corollary \ref{co:GN} but is proved using the different ideas
introduced in \cite{NP}.  Consider a finite set $\cl$ of linear
subspaces of $X$ which is stable under intersections and contains
$\{0\}$ and $X$. Let us denote $\cR_\cl(X)$ the $C^*$-subalgebra of
$\cR(X)$ generated by the $C(\oXY)$ with $Y\in\cl$. Then $\cR_\cl(X)$
is stable under translations so the crossed product
$\rr_\cl(X)=\cR_\cl(X)\rtimes X$ is a well defined $C^*$-subalgebra of
$\rr(X)$. Then Theorem 1.1 from \cite{MNP} says that the assertion of
Corollary \ref{co:GN} holds for the self-adjoint operators $H$
affiliated to $\rr_\cl(X)$.

The rest of this subsection is devoted to the proof of the last
assertion of Theorem \ref{th:GN}. The main ingredient of the proof is
a sufficiently explicit description of the character space
$\cR(X)^\dagger=\sigma(\cR(X))\setminus X$. In \cite{GN} one may find
a quite detailed but rather involved description of $\sigma(\cR(X))$;
here we will use a simpler slightly different construction.   

The character space of the algebra $C(\oX)$ is $\oX$ and each
$\alpha\in\mbs_X$ defines a character of $C(\oX)$, that we also denote
by $\alpha$, namely $\alpha(u)=u(\alpha)=\lim_{r\to\infty}u(ra)$ if
$a\in\alpha$. In fact $\alpha$ \emph{extends naturally to a character
  of} $\cR(X)$:
\begin{equation}\label{eq:charalpha}
\alpha(u)\doteq\lim_{r\to\infty}u(ra) \quad\forall u\in\cR(X). 
\end{equation}
Indeed, it suffices to prove that the limit above exists if
$u\in C(\oXY)$ for any subspace $Y$. But this is clear: since
$u\equiv u\circ\pi_Y$, if $\alpha\subset Y$ then $u(ra)=u(0)$ and if
$\alpha\not\subset Y$ then $u(ra)=u(r\pi_Y(a))$ hence if we set
$\beta=\pi_Y(\alpha)\in\mbs_{X/Y}$ we get
$\lim_{r\to\infty}u(ra)=\beta(u)$. 

Thus we obviously get a canonical embedding
$\mbs_X\subset\cR(X)^\dagger$. Now let us compute the translation
morphism $\tau_\alpha$ associated to $\alpha\in\mbs_X$. From
\eqref{eq:mtrans} we get
\[
\tau_\alpha(u)(x)=\alpha(\tau_x(u))=\lim_{r\to\infty}\tau_x(u)(ra)
=\lim_{r\to\infty}u(x+ra)  
\]
hence we see that for each $u\in\cR(X)$ and $x\in X$ the limit
\begin{equation}\label{eq:talpha}
\tau_\alpha(u)(x)\doteq \lim_{r\to\infty}u(ra+x)
\quad\text{exists and is independent of } a
\end{equation}
which is the explicit expression for the translation morphism
associated to $\alpha$. 

To better understand the action of $\tau_\alpha$ we have to recall
more of the formalism introduced in \cite{GN}.  If $Y\subset X$ is a
linear subspace then the $C^*$-algebra $\cR(X/Y)$ associated to the
vector space $X/Y$ is well defined and, by using the embedding
$C_\rmb^\rmu(X/Y)\subset C_\rmb^\rmu(X)$, one may identify it with a
subalgebra of $\cR(X)$:
\begin{equation}\label{eq:crxy}
\cR(X/Y)=
\text{ $C^*$-algebra generated by the subalgebras } C(\oXZ)
\text{ with } Z\supset Y . 
\end{equation}
Let $[\alpha]$ be the one dimensional subspace generated by $\alpha$
and set $X/\alpha=X/[\alpha]$ to simplify the writing.  Then \emph{the
  function $\tau_\alpha(u)$ belongs to the subalgebra $\cR(X/\alpha)$
and}
\begin{equation}\label{eq:talp}
\text{\emph{the map }} \tau_\alpha:\cR(X)\to\cR(X/\alpha) 
\text{ \emph{is a surjective morphism and a projection}}.
\end{equation}
Thus the translation morphism $\tau_\alpha$ also has the property
$\tau_\alpha^2=\tau_\alpha$, more precisely $\tau_\alpha(u)=u$ if
and only if $ u\in\cR(X/\alpha)$.
Recall that the operation $\tau_\alpha(u)$ on the function $u$ should
be thought as a translation by the point at infinity
$\alpha\in\mbs_X$.  On the other hand, the operator $\tau_a$ of
translation by some $a\in X$ is an automorphism of $\cR(X)$ but never
a projection.

\begin{lemma}\label{lm:vchar}
  Let $\vkappa$ be a character in $\cR(X)^\dagger$.
\begin{compactenum}
\item There is a unique $\alpha\in\mbs_X$ such that
  $\vkappa|C(\oX)=\alpha$.  This $\alpha$ will be denoted
  $\alpha_\vkappa$.

\smallskip
\item 
The map $\cR(X)^\dagger\ni\vkappa\mapsto\alpha_\vkappa\in\mbs_X$
  is surjective.

\smallskip
\item Denote $\what\vkappa\in\sigma(\cR(X/\alpha_\vkappa))$ the
  restriction of $\vkappa$ to $\cR(X/\alpha_\vkappa)$. Then
  $\vkappa=\what\vkappa\tau_{\alpha_\vkappa}$.

\smallskip
\item Let $Y\subset X$ with
  $\alpha_\vkappa\not\subset Y$ and
  $\beta_\vkappa\doteq\pi_Y(\alpha_\vkappa)\in\mbs_{X/Y}$. Then
  $\vkappa|C(\oXY)=\beta_\vkappa$.

\smallskip
\item We have $\tau_\vkappa\tau_{\alpha_\vkappa}=\tau_\vkappa$ and
  $\tau_\vkappa=\tau_{\what\vkappa}\tau_{\alpha_\vkappa}$. 
\end{compactenum}
\end{lemma}

\begin{proof}
  To prove (1), note that the restriction of $\vkappa$ to $C(\oX)$ is
  a character of $C(\oX)$ hence it must be the evaluation at a
  uniquely determined point of $\oX$. If this point is some $a\in X$,
  and since $C_0(X)$ is an ideal in $\cR(X)$, then we will have
  $\vkappa=\chi_a$ which is not possible because $\vkappa\nin X$.
  Hence the point belongs to $\mbs_X$. On the other hand, any point of
  $\mbs_X$ can be obtained in this way because each character of
  $C(\oX)$ is the restriction of a character of $\cR(X)$; this proves
  (2). Now let us prove (3). Note first that for any $\alpha\in\mbs_X$
  the restriction $\tau_\alpha|C(\oX)$ is just the character $\alpha$
  of $C(\oX)$. Then $\vkappa$ and $\what\vkappa\tau_{\alpha_\vkappa}$
  are characters of $\cR(X)$ whose restrictions to $C(\oX)$ and
  $\cR(X/\alpha)$ are $\alpha$ and $\what\vkappa$, so are equal, hence
  $\vkappa=\what\vkappa\tau_{\alpha_\vkappa}$ as a consequence of
  \cite[Cor.\ 6.8]{GN}.  The assertion (4) follows immediately from
  \cite[Lem.\ 6.7]{GN}.

  Finally, we prove (5). To simplify notations we write
  $\alpha\equiv\alpha_\vkappa$ and first prove
  $\tau_\vkappa\tau_{\alpha}=\tau_\vkappa$. The equality holds on
  $\cR(X/\alpha)$ because $\tau_\alpha$ is the identity on this
  subalgebra. Thus it remains to show that $\tau_\vkappa\tau_{\alpha}$
  and $\tau_\vkappa$ are equal on any $C(\oXY)$ if
  $\alpha\not\subset Y$. We saw before that
  $\tau_\alpha(u)=u(\pi_Y(\alpha))$ if $u\in C(\oXY)$ hence
  $\tau_\alpha|C(\oXY)$ is the character $\beta$ associated to the
  point $\beta=\pi_Y(\alpha)\in\mbs_{X/Y}$. Since $\tau_\vkappa$ is a
  unital morphism we see that
  $\tau_\vkappa\tau_{\alpha}|C(\oXY)=\beta$. It remains to show that
  we also have $\tau_\vkappa|C(\oXY)=\beta$. But the definition from
  \S\ref{ss:alg} gives
  $\tau_\vkappa(u)(x)=\vkappa(\tau_x(u))=(\tau_x(u))(\beta) =u(\beta)$
  where we also used the point (4) above; this proves the assertion.

  It remains to prove the relation
  $\tau_\vkappa=\tau_{\what\vkappa}\tau_{\alpha}$.  Here
  $\tau_{\what\vkappa}$ is the endomorphism of the algebra
  $\cR(X/\alpha)$ associated to the character $\what\vkappa$ by the
  rule $\tau_{\what\vkappa}(v)(y)=\what\vkappa(\tau_y(v))$ for
  $v\in\cR(X/\alpha)$ and $y\in X/\alpha$, cf.\
  \S\ref{ss:alg}. Since $\what\vkappa$ is a restriction of
  $\vkappa$ we get $\tau_{\what\vkappa}(v)(y)=\vkappa(\tau_y(v))$. Let
  us denote $\pi_\alpha$ the canonical surjection $X\to X/\alpha$ and
  recall that we decided to identify a function $w$ on $X/\alpha$ with
  the function $w\circ\pi_\alpha$ on $X$. Then the preceding identity
  can be written $\tau_{\what\vkappa}(v)(x)=\vkappa(\tau_x(v))$ for
  all $x\in X$ (in fact $y=\pi_\alpha(x)$ and $\tau_y(v)=\tau_x(v)$).
  If $u\in\cR(X)$ then $\tau_{\alpha}(u)\in\cR(X/\alpha)$ hence we
  obtain
  $\tau_{\what\vkappa}(\tau_{\alpha}(u))(x)=\vkappa(\tau_x(\tau_{\alpha}(u)))$
  for all $x\in X$.  But
  $\vkappa(\tau_x(\tau_{\alpha}(u)))=\tau_\vkappa(\tau_{\alpha}(u))(x)$
  so we get
  $\tau_{\what\vkappa}(\tau_{\alpha}(u))(x)=\tau_\vkappa(\tau_{\alpha}(u))(x)$
  for all $x\in X$ and $u\in\cR(X)$ so that
  $\tau_{\what\vkappa}\tau_{\alpha}=\tau_\vkappa\tau_{\alpha}=\tau_\vkappa$.
\end{proof}

\begin{corollary}\label{co:vchar}
  The map $\vkappa\mapsto(\alpha_\vkappa,\what\vkappa)$ induces a
  bijection of $\cR(X)^\dagger$ onto the disjoint union of the
  character spaces $\sigma(\cR(X/\alpha))$, where $\alpha$ runs over
  $\mbs_X$. The inverse map is $(\alpha,\chi)\mapsto\chi\tau_\alpha$.
  Thus we may identify
\begin{equation}\label{eq:vchar}
\cR(X)^\dagger \simeq\amalg_{\alpha\in\mbs_X}\sigma(\cR(X/\alpha)).
\end{equation} 
\end{corollary}

From Theorem \ref{th:GIA}, relation \eqref{eq:vchar}, and (5) of Lemma
\ref{lm:vchar}, we obtain for any $A\in\rr(X)$:
\begin{equation}\label{eq:hnb}
\spe (A)= \ccup_{\vkappa\in\cR(X)^\dagger} \spec(\tau_\vkappa(A))
=\ccup_{\vkappa\in\cR(X)^\dagger} 
\spec(\tau_{\what\vkappa}\tau_{\alpha_\vkappa}(A)) .
\end{equation} 
The maps $\tau_{\what\vkappa}:\rr(X)\to\rr(X/\alpha_\vkappa)$ are
morphisms, so
$\spec(\tau_{\what\vkappa}\tau_{\alpha_\vkappa}(A))\subset
\spec(\tau_{\alpha_\vkappa}(A))$.
Note that we have equality here if $\vkappa$ is such that
$\what\vkappa\in X/\alpha_\vkappa$. Then by using (2) of Lemma
\ref{lm:vchar} we obtain
\begin{equation}\label{eq:hnbb}
\spe (A)= \ccup_{\alpha\in\mbs_X} \spec(\tau_\alpha(A)) .
\end{equation} 
It remains to give a convenient expression to $\tau_\alpha(A)$.
Consider first the translation morphism
$\tau_\alpha:\cR(X)\to\cR(X/\alpha)$ associated to $\alpha$. If we
embed $\cR(X)\subset\rb(X)$ as a an algebra of multiplication
operators, then the definition \eqref{eq:talpha} implies 
\begin{equation}\label{eq:talphaop}
\tau_\alpha(u(q))=\slim_{r\to\infty} \e^{\rmi rap}u(q)\e^{-\rmi rap}.
\end{equation}
This  clearly gives for the induced morphism
$\tau_\alpha:\rr(X)\to\rr(X/\alpha)$ the formula
\begin{equation}\label{eq:talphaOP}
\tau_\alpha(A)=\slim_{r\to\infty} \e^{\rmi rap}A\e^{-\rmi rap} \quad
\forall A\in\rr(X).
\end{equation}
Finally, by using \eqref{eq:hnbb} and \eqref{eq:talphaOP} we get the
last assertion of Theorem \ref{th:GN}.

\section{Appendix}\label{s:app}
\protect\setcounter{equation}{0}

A filter on $X$ is a set $\vkappa$ of subsets which does not contain
the empty set, is stable under intersections, and satisfies
$T\supset S\in\vkappa\Rightarrow T\in\vkappa$.  If $f:X\to Y$ where
$Y$ is a topological space and if $y\in Y$ then
$\lim_{x\to\vkappa}f(x)=y$ means that $f^{-1}(V)\in\vkappa$ for any
neighborhood $V$ of $y$. A filter is finer than the Fr\'echet filter
if it contains the complements of bounded subsets.  Un ultrafilter is
a filter which is not included in any other filter.

The character space of $C_\rmb^\rmu(X)$ is called \emph{uniform
  compactification} of $X$ and is a quotient of the
\emph{Stone-\v{C}ech compactification} $\beta X$ of $X$, which is the
character space of the $C^*$-algebra $C_\rmb(X)$ of bounded continuous
functions on $X$. In turn, $\beta X$ is a quotient of the
Stone-\v{C}ech compactification $\gamma X$ of the \emph{discrete}
space $X$, which is the character space of the $C^*$-algebra of all
bounded functions on $X$, and its elements are the ultrafilters on
$X$. The analogue of $\ca^\dagger$ in this context is the set of
ultrafilters $\vkappa$ finer than the Fr\'echet filter. We denote
$ X^\dagger$ the set of all such ultrafilters.

Now we restate Theorem \ref{th:rb} without using the notion of
ultrafilters and thus make the connection with the results from
\cite{RRS0,RRS,LaS,LS} and references there.  Following
\cite{RRS0,RRS}, we use the \emph{operator spectrum of $A$}:
\begin{equation}\label{eq:opspec}
\sigma_\op(A)=\{B\in\rb \mid B=\ulim_{n\to\infty} \tau_{x_n}(A)
  \text{ for a sequence } \{x_n\} \text{ with } |x_n|\to\infty\}.  
\end{equation}

\begin{theorem}\label{th:classic}
  If $A\in\re$ then the set $\{\tau_x(A)\mid x\in X\}$ is relatively
  compact in $\re_\loc$ and $\sigma_{\op}(A)$ is a compact subset of
  $\re_\loc$. We have
  \begin{equation}\label{eq:classic}
    \spe(A)=\ccup_{B\in\sigma_{\op}(A)} \spec(B).
  \end{equation}
\end{theorem}

\begin{proof}
  It suffices to show that
  $\sigma_\op(A)=\{A_\vkappa\mid\vkappa\in X^\dagger\}$, cf.\ Theorem
  \ref{th:GIold} and Lemma \ref{lm:topology}. If
  $\vkappa\in X^\dagger$ then $A_\vkappa=\ulim_{x\to\vkappa}\tau_x(A)$
  by \eqref{eq:tchar}.  Since $\rb_\loc$ is metrizable, there is a
  metric $d$ which defines its topology and then for each integer
  $n>0$ there is $x_n$ with $|x_n|>n$ and
  $d(\tau_{x_n}(A),A_\vkappa)<1/n$, hence
  $A_\vkappa\in\sigma_\op(A)$. Reciprocally, if $B\in\sigma_\op(A)$
  then there is  a sequence $\{x_n\}$  with $|x_n|\to\infty$ such that
  $B=\ulim_{n\to\infty} \tau_{x_n}(A)$. Let $\vkappa$ be the set of
  subsets $F$ of $X$ with the property: there is $N$ such that $x_n\in
  F$ for all $n>N$. Clearly $\vkappa$ is an ultrafilter finer than the
  Fr\'echet filter and $B=\ulim_{x\to\vkappa}\tau_x(A)=A_\vkappa$.
\end{proof}

\end{document}